%% file: main.tex
\documentclass{iopjournal_mod}

\usepackage[T1]{fontenc}
\usepackage[utf8]{inputenc}
\usepackage{lmodern}
\usepackage[english]{babel}
\usepackage{caption}
\usepackage{subcaption} 
\usepackage{appendix}
\usepackage{paralist}

\usepackage{algorithm}
\usepackage{algpseudocode}
\algrenewcommand\alglinenumber[1]{\scriptsize\itshape\bfseries #1}
\makeatletter 
 
\@addtoreset{algorithm}{section} 
\makeatother
\input{defineverticalline.tex}

\usepackage{hyperref}
\usepackage{writemaths}
\usepackage{bbm}
\input{myglossary.tex}

\newcommand{\Capprox}{\ensuremath{\widehat{\bs{C}}_{P}}}

\begin{document}
\articletype{Paper}

\title{Covariance-Informed Subspace: an Adaptive Gradient-Free Input Dimension Reduction Method for Bayesian Inference}
\author{Nadège Polette$^{1,2,*}$\orcid{0009-0002-1276-5249}, Olivier Le Maître$^3$\orcid{0000-0002-3811-7787}, Pierre Sochala$^{1,4}$\orcid{000-0002-8223-9335} and Alexandrine Gesret$^{2}$\orcid{0000-0002-6828-392X}}
\affil{$^1$CEA DAM DIF, F-91297 Arpajon, France}
\affil{$^2$Geosciences center, Mines Paris PSL, 77300 Fontainebleau, France}
\affil{$^3$CMAP, CNRS, Ecole Polytechnique, 91477 Palaiseau, France}
\affil{$^4$Autorité de Radioprotection et de Sureté Nucléaire (ASNR), PSE-ENV/SCAN, F-92260, Fontenay aux Roses, France}
\affil{$^*$Author to whom any correspondence should be addressed.}
\email{nadege.polette@cea.fr}

\keywords{Dimension reduction, gradient-free, covariance estimation, iterative algorithm, field inference}

\begin{abstract}
This paper addresses the challenge of dimension reduction (DR) in Bayesian inference of high-resolution two- or three-dimensional fields, where a priori parametrizations require a large number of terms. The underlying idea is common to state-of-the-art methods in which the parameter space is decomposed into two subspaces, one informed by the likelihood and one constrained by the prior. DR techniques generally use gradient information from the log-likelihood to derive the corresponding subspaces. However, the gradient may be unavailable or expensive to compute accurately, for instance in the case of simulation-based inference. Inspired by approaches based on likelihood-informed subspaces, we develop a new DR method tailored for settings where gradient computation is not feasible. More specifically, we propose a gradient-free indicator for determining whether a direction is informed by the data. This indicator is derived from the posterior-to-prior covariance ratio introduced in Spantini et al.~(2015). We show that, in the linear Gaussian case, this indicator combined with an approximate likelihood leads to a better posterior approximation. The method is then extended to nonlinear cases, and strategies to approximate the posterior covariance are detailed. We demonstrate the effectiveness of this DR through two high-dimensional inference problems arising from groundwater and atmospheric applications.
\end{abstract}
 
\input{chapter}

\roles{Nadège Polette - Conceptualization, Methodology, Validation, Writing - original draft, Writing - review \& editing; Olivier Le Maître - Conceptualization, Supervision, Writing - review \& editing; Pierre Sochala - Conceptualization, Supervision, Writing - review \& editing; Alexandrine Gesret - Supervision, Writing - review \& editing}
\data{The data for the GOMOS test case will be made available. The groundwater test case is based on a finite-element solver that is not public. Supplementary data will be made available on request.}

\appendix
\value{tocdepth}=0
\input{appendixA}

\input{appendixB}

{\small
\bibliographystyle{plain}
\bibliography{bibliography}
}
\end{document}

%% file: myglossary.tex
\usepackage[acronyms]{glossaries}
\makeglossaries

\newacronym{as}{AS}{\textit{Active Subspace}}
\newacronym{blg}{BLG}{\textit{Bayesian Linear Gaussian}}
\newacronym{cis}{CIS}{\textit{Covariance-Informed Subspace}}
\newacronym{cissmc}{CIS-SMC}{\textit{Covariance-Informed Subspace coupled with Sequential Monte Carlo}}
\newacronym{dr}{DR}{\textit{dimension reduction}}
\newacronym{fd}{FD}{\textit{finite difference}}
\newacronym{gis}{GIS}{\textit{Gradient-based Informed Subspace}}
\newacronym{gomos}{GOMOS}{\textit{Global Ozone Monitoring by Occultation of Stars}}
\newacronym{kl}{KL}{\textit{Karhunen--Loève}}
\newacronym{kld}{KL divergence}{\textit{Kullback--Leibler divergence}}
\newacronym{lis}{LIS}{\textit{Likelihood-Informed Subspace}}
\newacronym{mc}{MC}{\textit{Monte Carlo}}
\newacronym{mcmc}{MCMC}{\textit{Markov chain Monte Carlo}}
\newacronym{mh}{MH}{\textit{Metropolis--Hastings}}
\newacronym{pcn}{pCN}{\textit{preconditioned Cranck Nicholson}}
\newacronym{smc}{SMC}{\textit{Sequential Monte Carlo}}
\newacronym{wmc}{wMC}{\textit{weighted Monte Carlo}}

%% file: chapter.tex
\input{section1}
\input{section2}
\input{section3}
\input{section4}
\input{section5}
\input{section6}

%% file: section1.tex
\section{Introduction}
This paper presents a Bayesian approach for inferring an unknown field of interest using indirect observations.
Even with suitable finite-dimensional parametrizations, such as a nodal or a modal decomposition~\cite{kaipio2005, marzouk2009, polette2025}, achieving an accurate approximation is challenging as it typically involves several hundred parameters.
The number of parameters naturally increases with the dimensionality of the physical space and when the field covers a large domain and/or exhibits fine scale variations.
The goal of this paper is to develop an iterative \gls{dr} method to represent the field of interest with as few parameters as possible, without compromising the inference accuracy. 
Exploiting the sparsity of the data, this \gls{dr} relies on the fact that the likelihood is generally informative only on a small-dimensional subspace of the parameter space~\cite{constantine2016}. 
This assumption allows the parameter space to be viewed as the union of two orthogonal sets, one mainly informed by the likelihood, and the other constrained by the prior.
By sampling only the informed subspace, the computational cost of the inference process can be substantially reduced.

Several linear methods have been proposed to decompose the parameter space.
Relying on the Hessian of the log-likelihood, the \gls{lis}~\cite{cui2014} was first proposed for the \gls{blg} case in~\cite{spantini2015}.
The \gls{as}~\cite{constantine2015} was applied to the Bayesian framework by setting the log-likelihood function as the function of interest and identifying its most informative gradient directions~\cite{constantine2016}.
In both methods, the curvature of the log-posterior density along the subspace directions is more constrained by the log-likelihood than by the prior. 
The Fisher information matrix is also used in~\cite{cui2022a}, where different constructions for the nonlinear case are compared, and in~\cite{zahm2022}, where a bound of the approximation error is derived and the general form of the likelihood approximation is shown to be optimal (with respect to various divergences). Some recent works also tackle the coupled \gls{dr} problem, where the objective is to jointly reduce the dimension of the inputs and the outputs~\cite{lieberman2010, giraldi2018, vohra2020, chen2025}. In this paper, we focus on the input \gls{dr}, since the computational cost of the inference depends exponentially on the input dimension due to the curse of dimensionality, while it only depends linearly on the output dimension in our applications.
Nonlinear \gls{dr} techniques such as Kernel methods~\cite{scholkopf1998, gedon2023} or variational autoencoders~\cite{kingma2013,goh2022,tonini2025} could also be envisioned. 
While this could be efficient, we instead focus on linear methods to derive proofs of optimality and error bounds. 
In practice, we observe in our applications that a linear informed subspace is well suited to approximate the posterior.
The theoretically guaranteed linear methods rely on the computation of an approximated Hessian of the log-likelihood, which can be expensive or even unavailable as pointed out in~\cite{brennan2022}. 

This study presents a new construction for the linear projection operator that defines the informed subspace. The general idea is inspired by the work of~\cite{spantini2015} which states that approximating the posterior covariance is equivalent to approximating its inverse in the \gls{blg} case. Instead of relying on the Hessian of the log-likelihood, the approximation of the inverse posterior covariance involves the ratio of the posterior and the prior covariances. The low-dimensional subspace is defined as the eigendirections associated with the smallest eigenvalues of the pencil composed of the posterior and the prior covariances. 
Intuitively, this approach involves identifying the informed subspace as the set of directions along which the posterior covariance is most significantly reduced relative to the prior covariance. Combined with the approximated likelihood formulation of~\cite{zahm2022}, the proposed \gls{cis} does not require any gradient information and is optimal with respect to the Förstner distance between the true posterior covariance and the approximated one in the \gls{blg} case. The subspace construction requires estimating the unknown posterior covariance. We rely on a \gls{wmc} estimator, updated through an iterative scheme. For nonlinear inverse problems, we propose to generalize this construction based on the covariance ratio. Bounds are derived to control the approximation error and an adaptive tempered framework based on \gls{smc} is proposed to improve the \gls{wmc} estimator capabilities.
We illustrate on two examples that the gradient-free indicator is sufficient to identify an informed subspace similar to the one obtained with state-of-the-art gradient-based methods. We also demonstrate the capability of the adaptive tempered framework to overcome weight degeneracy in the case of a small posterior variance.

The outline of the paper is as follows. 
In Section~2, the mathematical notations and the Bayesian problem are introduced. The approximation form we consider is proven to be optimal in the \gls{blg} case. In Section~3, the gradient-free projector using the covariance ratio is introduced along with the corresponding sampling algorithms. Sections~4 and~5 demonstrate its ability to efficiently reduce the problem dimensionality on two test cases. The first one is a 100-dimensional groundwater problem, assessing the ability of the method to recover results close to gradient-based approaches. The second one is a real application to \gls{gomos} observations, challenging the method with a posterior that is numerically difficult to sample. Conclusions and perspectives are drawn in Section~6.

%% file: section2.tex
\section{Dimension reduction in Bayesian inverse problems}\label{section:dr}
The Bayesian inverse problem formulation is detailed in Section~\ref{subsection:dr:formulation}. Then, we introduce the \gls{dr} framework in Section~\ref{subsection:dr:framework}. The construction of the approximated likelihood is explained in Section~\ref{subsection:dr:approx-likelihood}, as well as the construction of the projector in Section~\ref{subsection:dr:projector}. Finally, in Section~\ref{subsection:dr:blg}, a new result that combines state-of-the-art elements is presented for the \gls{blg} case.

\subsection{Bayesian inverse problem formulation}\label{subsection:dr:formulation}
The inverse problem consists in estimating the parameters~$\bx \in \R^{n}$ from data~$\by \in \R^{m}$, assuming that
\begin{equation}
\by = \mc{F}(\bx) + \bs{\varepsilon},
\end{equation}
where~$\mc{F} : \R^n \rightarrow \R^m$ is a model-based predictor composed of a physical model and an observation operator. It is called the forward model in the following. The observation noise~$\bs{\varepsilon}$ is assumed to be Gaussian~$\bs{\varepsilon} \sim \mc{N}(0,\bs{C}_\varepsilon)$. Denoting~$X$ (resp.~$Y$) a random vector taking values in~$\R^n$ (resp.~$\R^m$), the likelihood of the data given the parameters is defined as 
\begin{equation}
\mc{L}(X;Y) = \inv{(2\pi)^{m/2}\mathrm{det}\left(\bs{C}_\varepsilon\right)^{1/2}}\exp \left(\frac{-1}{2}(\mc{F}(X)-Y)^\top \bs{C}_\varepsilon^{-1}(\mc{F}(X)-Y) \right).
\end{equation}
Given a prior~$X \sim \pi$, the posterior probability of~$X$ given~$Y$ is, according to Bayes' rule,
\begin{equation}\label{eq:post-distrib}
P(X|Y) = \mc{L}(X;Y)\pi(X)/\ev,
\end{equation}
where~$\ev$ is a normalization constant called the \textit{evidence} and is expressed as
\begin{equation}
    \ev = \Int_{\R^n} \mc{L}(\bx;Y)\pi(\bx) d\bx.
\end{equation}
For simplicity, we write~$\mc{L}(X)$ and~$P(X)$ instead of~$\mc{L}(X;Y)$ and~$P(X|Y)$ when~$Y$ is understood from context. A closed-form expression for the posterior distribution can be derived in the \gls{blg} setting introduced in Example~\ref{ex:blg}.
\begin{example}[Bayesian Linear Gaussian (BLG) case]\label{ex:blg}
    Let~$\mc{F}$ be a linear forward model 
    \begin{equation}
        Y = \bs{F}X + \bs{\varepsilon},
    \end{equation}
    with~$\bs{F} \in \R^{m \times n}$ and a prior~$\pi$ assumed to be Gaussian:~$X \sim \mc{N}(\bs{\mu}_{\prior}, \bs{C}_{\prior})$. Under these assumptions, the posterior distribution is also Gaussian, as a product of two Gaussian distributions,
    \begin{equation}\label{eq:blg}
    X|Y \sim \mc{N}(\bs{\mu}_P(Y), \bs{C}_P), \text{ with } \left\{ \begin{array}{l} 
        \bs{\mu}_P(Y) = \bs{C}_P \left[ \bs{F}^\top \bs{C}_\varepsilon^{-1} Y + \bs{C}_{\prior}^{-1}\bs{\mu}_{\prior}\right],
        \\
        \bs{C}_P = \left( \bs{F}^\top \bs{C}_\varepsilon^{-1} \bs{F} + \bs{C}_{\prior}^{-1} \right)^{-1}. \end{array} \right. 
    \end{equation}
\end{example}

\subsection{Dimension reduction framework}\label{subsection:dr:framework}
Evaluating the posterior distribution~$P$ is generally challenging and requires the use of sampling techniques such as \gls{mcmc}. However, these algorithms become extremely expensive when the dimension~$n$ of the parameter space increases. In order to mitigate the computational cost, we aim at sampling in a reduced parameter space~$\R^r$ where~$r\ll n$ by relying on an approximate posterior~$P_{\Pi_r}$. The underlying assumption is that the likelihood is informative only on a low-dimensional subspace, since the data are sparse and/or noisy. We define two reduced parameters~$\bs{z}_r \in \R^r$ and~$\bs{z}_\perp \in \R^{n-r}$ such that
\begin{equation}\label{eq:red-form}
\bx = \bx_r + \bx_\perp, \quad \text{with } \left\{ \begin{array}{l}
    \bx_r = \bs{V}_r \bs{z}_r = \bs{\Pi}_r \bx,\\
    \bx_\perp = \bs{V}_\perp \bs{z}_\perp = (\mathrm{I}_n - \bs{\Pi}_r)\bx.
\end{array}\right.
\end{equation}
where~$\bs{V}_r \in \R^{n \times r}$ (resp.~$\bs{V}_\perp \in \R^{n \times (n-r)}$), and~$\bs{\Pi}_r$ is a linear projection operator such that~$\mathrm{rank}(\bs{\Pi}_r) = r$.
The spaces~$\mathrm{Im}(\bs{\Pi}_r)$ and $\mathrm{Ker}(\bs{\Pi}_r)$ are respectively called the informed and non-informed subspaces. \\

Using decomposition~\eqref{eq:red-form} and denoting the corresponding random vectors using upper cases, we define the approximate posterior as
\begin{align}\label{eq:post-distrib-approx}
P_{\Pi_r}(X|Y) &\equiv \mc{L}_{\Pi_r}(Z_r;Y)\pi(Z_r,Z_\perp)/\pi_{\mathrm{data},\Pi_r}(Y) \notag \\
&= \mc{L}_{\Pi_r}(Z_r;Y)\pi_{r}(Z_r)\pi_{\perp|r}(Z_\perp |Z_r)/\pi_{\mathrm{data},\Pi_r}(Y),
\end{align}
where~$\pi_r$ denotes the marginal prior probability distribution of~$Z_r$,~$\pi_{\perp|r}$ denotes the conditional prior probability distribution of~$Z_\perp$ given~$Z_r$ and~$\mc{L}_{\Pi_r}$ is an approximation of the likelihood depending only on~$Z_r$ which is defined in the next section. The prior probability definitions also depend on the projection, even if it does not appear in the notation for brevity. Since the approximate likelihood only depends on~$Z_r \in \R^r$, one can only sample this lower dimensional parameter space, and draw~$Z_\perp$ according to its conditional prior. Note that it is equivalent to define the approximate posterior using the decomposition involving~$X_r$ and~$X_\perp$, although the \gls{dr} is not as clearly apparent in that formulation.

\subsection{Approximate likelihood construction}\label{subsection:dr:approx-likelihood}
This section is devoted to the construction of an appropriate approximation to the likelihood function that depends only on the reduced parameters~$Z_r$. The objective is to derive a formulation for~$\mc{L}_{\Pi_r}$ that preserves the essential information while enabling efficient computation. While the initial \gls{lis} just considers~$\mc{L}_{\Pi_r}(Z_r;Y) = \mc{L}(X_r;Y)$, various choices have been tested and analyzed in the literature~\cite{cui2015,cui2022a}. We define the approximate likelihood~$\mc{L}_{\Pi_r}(Z_r;Y)$ as the expectation of the likelihood along the non-informed subspace~\cite{cui2022a,zahm2022}, as detailed in Definition~\ref{def:approx-likelihood}. 
\begin{definition}[Approximate likelihood~$\mc{L}_{\Pi_r}$]\label{def:approx-likelihood}
    The approximate likelihood is the likelihood marginalized over the non-informed subspace,
\begin{equation}\label{eq:approx-likelihood}
    \mc{L}_{\Pi_r}(Z_r;Y) \coloneqq \E_{\pi_{\perp|r}}\left( \mc{L}(X;Y) | Z_r \right) = \Int_{\R^{n-r}} \mc{L}(\bs{V}_r Z_r + \bs{V}_\perp \bs{z}_\perp;Y)\pi_{\perp}(\bs{z}_\perp|Z_r) d\bs{z}_\perp.
\end{equation}
\end{definition}
Definition~\ref{def:approx-likelihood} is optimal with respect to the~$L^2_\pi$-norm, the \gls{kld}~\cite[Section~2]{zahm2022} and all expected Bregman divergences~\cite{banerjee2005}. With this formulation,~$\ev = \pi_{\mathrm{data},\Pi_r}(Y)$.
Since the normalization constant is the same for the posterior and its approximation, it holds  
\begin{equation}\label{eq:post-approx-decomp}
P_{\Pi_r}(X|Y) = \mc{L}_{\Pi_r}(Z_r;Y)\pi_{r}(Z_r)\pi_{\perp|r}(Z_\perp|Z_r) / \ev = P_{\Pi_r,r}(Z_r|Y)\pi_{\perp}(Z_\perp|Z_r),
\end{equation}
where the notation~$\cdot_{r}$ indicates the marginal probability distribution of~$Z_r$. This rewriting clearly distinguishes the reduced random variable~$Z_r \in \R^r$ which is inferred, from~$Z_\perp$ which is sampled according to the conditional prior. 
The equality of the marginals is given in Lemma~\ref{lemma:marg-post}.
\begin{lemma}[Marginal equality]\label{lemma:marg-post}
    For any posterior distribution~$P$~\eqref{eq:post-distrib}, for any posterior approximation~$P_{\Pi_r}$~\eqref{eq:post-distrib-approx}, the approximate posterior marginal of~$Z_r$ is equal to its true value,
    \begin{equation}
    P_{\Pi_r,r}(Z_r|Y) = \Int_{\R^{n-r}} P(\bs{V}_r Z_r + \bs{V}_\perp \bs{z}_\perp | Y) d\bs{z}_\perp = P_r(Z_r|Y).
    \end{equation}
\end{lemma} 
The particular expression in the \gls{blg} case is detailed below. 
\begin{example}[Approximate posterior in the BLG case]\label{ex:approx-post-blg}
    In the \gls{blg} case, the Gaussian structure allows us to derive explicit expressions for all components of our approximation. Since~$P$ is a Gaussian distribution, its marginal is also Gaussian and according to Lemma~\ref{lemma:marg-post}, the approximate posterior marginal is equal to its true value. Therefore,~$\bs{V}_r Z_r | Y \stackrel{P_{\Pi_r, r}}{\sim} \mc{N}(\bs{\Pi}_r \bs{\mu}_P, \bs{\Pi}_r \bs{C}_P \bs{\Pi}_r^\top)$. 
Moreover, if~$Z_\perp$ is independent of~$Z_r$,~$X = \bs{V}_r Z_r + \bs{V}_\perp Z_\perp$ is the sum of two independent Gaussian variables. Since~$Z_\perp$ is also a transformation of a Gaussian variable,~$\bs{V}_\perp Z_\perp | Y \stackrel{\prior}{\sim} \mc{N}((\mathrm{I}_n - \bs{\Pi}_r)\bs{\mu}_{\prior}, (\mathrm{I}_n - \bs{\Pi}_r) \bs{C}_{\prior} (\mathrm{I}_n - \bs{\Pi}_r^\top))$, so that 
 \begin{equation}\label{eq:approx-post-blg}
    X|Y \stackrel{P_{\Pi_r}}{\sim} \mc{N}(\widehat{\bs{\mu}}_{P}, \Capprox), \text{ with } \left\{ \begin{array}{l} 
        \widehat{\bs{\mu}}_{P} = \bs{\Pi}_r \bs{\mu}_P + (\mathrm{I}_n - \bs{\Pi}_r) \bs{\mu}_{\prior},
        \\
        \Capprox = \bs{\Pi}_r \bs{C}_P \bs{\Pi}_r^\top + (\mathrm{I}_n - \bs{\Pi}_r) \bs{C}_{\prior} (\mathrm{I}_n - \bs{\Pi}_r^\top). \end{array} \right.
\end{equation}
\end{example}

In practice, analytically computing the expectation in Definition~\ref{def:approx-likelihood} is often intractable, hence requiring a \gls{mc} approximation~\cite{cui2022a},
\begin{equation}\label{eq:approx-likelihood-mc}
    \mc{L}_{\Pi_r}(Z_r) \simeq \mc{L}_{\Pi_r}^{(N)}(Z_r) := \inv{N}\suml_{i=1}^{N} \mc{L}(\bs{V}_r Z_r+ \bs{V}_\perp \bs{z}^{(i)}; Y), \text{ with } \bs{z}^{(i)} \sim \pi_{\perp|r}.
\end{equation}
The impact of this numerical approximation on the quality of the posterior approximation is analyzed in~\cite[Theorem~3.1 and its proof]{cui2022a}. 

\subsection{Projector construction}\label{subsection:dr:projector}
Once the approximate likelihood is defined, the next step is to determine a projection operator~$\bs{\Pi}_r$ which characterizes the reduced subspace. The choice of the projector is particularly important, as it determines which directions in the parameter space retain information from the data and which can be safely marginalized out. In the \gls{blg} case, approximating~$P$ is equivalent to approximating~$\bs{\mu}_P(Y)$ and~$\bs{C}_P$, and~\cite{spantini2015} proposes an optimal approximation of~$\bs{C}_P$ based on the Hessian of the negative log-likelihood~$\bs{H} = \bs{F}^\top \bs{C}_\varepsilon^{-1} \bs{F}$, where the optimality is achieved for a class of loss functions including the Förstner distance.
\begin{definition}[Förstner distance]\label{def:forstner}
The Förstner distance~\cite{forstner2003} is defined between two symmetric positive definite matrices $(\bs{A},\bs{B})$ as
\begin{equation}\label{eq:forstner}
    d_\mc{F}(\bs{A},\bs{B}) = \suml_{i=1}^n \left( \log \lambda_i \right)^2,
\end{equation}  
where $\{\lambda_i\}_{i=1}^n$ are the eigenvalues of the generalized eigenproblem associated to the pencil $(\bs{A},\bs{B})$, that is 
\begin{equation}
\bs{A}\bs{v}_i = \lambda_i \bs{B} \bs{v}_i,
\end{equation}
with $\bs{v}_i\in \R^n$ the eigenvector associated to $\lambda_i$. 
\end{definition} 
\noindent 
Using the expression~\eqref{eq:blg}, a natural class of matrices for approximating $\bs{C}_P$ is the set of negative semidefinite updates of $\bs{C}_\pi$, with a fixed maximum rank $r$, that lead to positive definite matrices,
\begin{equation}\label{eq:class-opt}
     \mc{M}_r = \{\bs{C}_\pi - \bs{K}\bs{K}^\top: \mathrm{rank}(\bs{K}) \leq r\}.
\end{equation}
The main results of~\cite{spantini2015} using $\mc{M}_r$ are recalled below.
\begin{theorem}[Spantini's theorem]\label{theo:spantini}
    Let~$(\delta_i, \bs{w}_i)_{i=1}^n$ be the generalized eigenpairs of the pencil~$(\bs{H}, \bs{C}_{\prior}^{-1})$ with the ordering $\delta_i > \delta_{i+1}$. An optimal approximation of~$\bs{C}_P$ considering an element of $\mc{M}_r$ (Eq.~\eqref{eq:class-opt}) is
\begin{equation}
    \label{eq:spantini}
    \widehat{\bs{C}}_P = \bs{C}_{\prior} - \bs{K}\bs{K}^\top, \text{ with } \bs{K}\bs{K}^\top = \suml_{i=1}^r \delta_i \left(1+\delta_i\right)^{-1}\bs{w}_i \bs{w}_i^\top.
\end{equation}
\end{theorem}
\noindent This approximation is optimal in terms of the Förstner distance~$d_\mc{F}(\widehat{\bs{C}}_P, \bs{C}_P)$, and consequently in terms of \gls{kld} if the mean is known. A corollary also defines an equivalent approximation of the posterior precision matrix~$\bs{C}_P^{-1}$.
\begin{corollary}[Spantini's corollary (precision matrix approximation)]\label{cor:spantini}
    Let $(\delta_i,\bs{w}_i)_{i=1}^n$ defined as in Theorem~\ref{theo:spantini}. An optimal approximation of~$\bs{C}_P^{-1}$ is given by
    \begin{equation}
        \label{eq:spantini-coro}
        \widehat{\bs{C}}_P^{-1} = \widehat{\bs{C}^{-1}}_P = \bs{C}_{\prior}^{-1} + \bs{U}\bs{U}^\top, \text{ with } \bs{U}\bs{U}^\top = \suml_{i=1}^r \delta_i \bs{u}_i \bs{u}_i^\top, \text{ and } \bs{u}_i = \bs{C}_{\prior}^{-1}\bs{w}_i.
    \end{equation}
    \noindent In addition,~$\left( 1/(1+\delta_i), \bs{u}_i\right)$ are the eigenpairs of the pencil~$(\bs{C}_P,\bs{C}_{\prior})$.
\end{corollary}
\noindent Finally, these two results are used to define a projector~$\bs{\Pi}_r$.
\begin{definition}[Spantini's projector]\label{def:spantini-proj}
The projector is written as
    \begin{equation}\label{eq:spantini-proj}
        \bs{\Pi}_r = \bs{W}_r \bs{U}_r^\top,
    \end{equation} 
    where the~$r$ columns of~$\bs{W}_r$ are the leading eigenvectors~$\{\bs{w}_i\}_{i=1}^r$ and the~$r$ columns of~$\bs{U}_r$ are the minor eigenvectors (smallest eigenvalues)~$\{\bs{u}_i\}_{i=1}^r$. 
\end{definition}
\noindent Combining the projector~\eqref{eq:spantini-proj} with the approximate likelihood~$\mc{L}_{\Pi_r}(Z_r; Y) = \mc{L}(X_r; Y)$ achieves optimality in the sense that \begin{inparaenum}[i)] \item the obtained approximate posterior covariance corresponds to~\eqref{eq:spantini}, and \item the approximation of the mean minimizes a weighted Bayes risk under squared-error loss\end{inparaenum}. 

Building on this linear case result, several extensions to nonlinear problems have been developed. The use of Hessian-based directions was first extended to the nonlinear case in~\cite{cui2014}, defining a \acrfull{lis} based on the expected value of a Gauss-Newton approximation of the log-likelihood Hessian over the posterior space. In~\cite{cui2022a}, other constructions for the projector are proposed using a Gram matrix of the gradient of the log-likelihood function. Finally, in~\cite{zahm2022}, this Gram matrix is computed along the posterior space, and the projector is constructed using the dominant eigenspace of the pencil composed of this matrix and the prior precision matrix (extension of Theorem~\ref{theo:spantini}). A bound on the \gls{kld} is obtained via logarithmic Sobolev inequalities and this bound is refined and extended on the broader class of Amari $\alpha$-divergences in~\cite{li2024}.

\subsection{Combination of previous results in the Bayesian Linear Gaussian case}\label{subsection:dr:blg}
We now combine the optimal approximate likelihood from Definition~\ref{def:approx-likelihood} with Spantini's optimal projector construction from Definition~\ref{def:spantini-proj}. Our contribution is to show that this pairing not only preserves the individual optimality properties but actually improves upon existing approaches. The following theorem establishes three key properties of this combined approach, namely the orthogonality of the informed and non-informed subspaces, the optimality of the approximate covariance and the improvement of the resulting \gls{kld}.
\begin{theorem}[BLG posterior approximation]\label{theo:blg-proj}
Using the projector~\eqref{eq:spantini-proj} and the approximate likelihood defined in Eq.~\eqref{eq:approx-likelihood} in the \gls{blg} case leads to the following results.
\begin{enumerate}
    \item $\bs{\Pi}_r$ and~$(\mathrm{I}_n-\bs{\Pi}_r)$ are~$\bs{C}_{\prior}$-orthogonal,
    \item the Förstner distance $d_\mc{F}(\widehat{\bs{C}}_P,\bs{C}_P)$ is minimal,
    \item the \gls{kld} between the posterior and its approximation is smaller than the one obtained by Spantini's formulation.
\end{enumerate}
\end{theorem}
The key insight behind this theorem is that this approximation naturally respects the geometry of the \gls{blg} problem. The projector~$\bs{\Pi}_r$ is constructed to align with the directions along which there is a major update from prior to posterior, while the approximate likelihood optimally integrates over the non-informed directions. 
\begin{proof}[Proof of Theorem~\ref{theo:blg-proj}]

\begin{enumerate}
\item Let~$(\lambda_i,\bs{u}_i)_{i=1}^n$ be the eigenpairs of~$(\bs{C}_P, \bs{C}_{\prior})$ with the ordering~$\lambda_i \geq \lambda_{i+1}$. It holds
\begin{equation}\label{eq:gen-eigen}
\bs{C}_P = \bs{C}_{\prior} \bs{U}\bs{\Lambda} \bs{U}^\top \text{ and } \bs{U}^\top \bs{C}_{\prior} \bs{U} = \mathrm{I}_n,
\end{equation}
where $\bs{U} = \left[\bs{u}_i\right]_{i=1}^n$ and $\bs{\Lambda} = \mathrm{diag}(\lambda_i)_{i=1}^n$.
From Corollary~\ref{cor:spantini}, we know that~$\bs{W}_r = \bs{C}_{\prior} \bs{U}_r$ where~$\bs{U}_r$ contains the first~$r$ columns of~$\bs{U}$, so that
\begin{equation}
\bs{\Pi}_r = \bs{W}_r \bs{U}_r^\top = \bs{C}_{\prior} \bs{U}_r \bs{U}_r^\top.
\end{equation}
Therefore, 
\begin{align*}
\bs{\Pi}_r \bs{C}_{\prior} (\mathrm{I}_n - \bs{\Pi}_r^\top) &= \bs{C}_{\prior} \bs{U}_r \bs{U}_r^\top \bs{C}_{\prior} - \bs{C}_{\prior} \bs{U}_r \bs{U}_r^\top \bs{C}_{\prior} \bs{U}_r \bs{U}_r^\top \bs{C}_{\prior} \\
&= \bs{C}_{\prior} \bs{U}_r \bs{U}_r^\top \bs{C}_{\prior} - \bs{C}_{\prior} \bs{U}_r \bs{U}_r^\top \bs{C}_{\prior} = \bs{0},
\end{align*}
which concludes the proof of orthogonality.\hfill \cqfd
\item Using the $\bs{C}_{\prior}$-orthogonality of $\bs{\Pi}_r$ and~$(\mathrm{I}_n-\bs{\Pi}_r)$, the independence assumption of Example~\ref{ex:approx-post-blg} is verified. Building upon this relation between~$\widehat{\bs{C}}_P$ and~$\bs{C}_P$, we want to compute their Förstner distance. First, we show that the eigenvalues of~$(\widehat{\bs{C}}_P, \bs{C}_P)$ can be split into two groups. 
Using the $\bs{C}_\pi$-orthogonality, the projection of the approximate posterior covariance in the \gls{blg} case simplifies to
\begin{align}\label{eq:proj-hatCp}
    \bs{\Pi}_r \widehat{\bs{C}}_P &\stackrel{\eqref{eq:approx-post-blg}}{=} \bs{\Pi}_r \bs{\Pi}_r \bs{C}_P \bs{\Pi}_r^\top + \bs{\Pi}_r (\mathrm{I}_n - \bs{\Pi}_r)\bs{C}_{\prior} (\mathrm{I}_n - \bs{\Pi}_r^\top) \notag \\
    &= \bs{\Pi}_r \bs{C}_P \bs{\Pi}_r^\top.
\end{align}
Let~$(\nu_i,\bs{v}_i)$ be an eigenpair of~$(\widehat{\bs{C}}_P,\bs{C}_P)$. By projecting on $\bs{\Pi}_r$, we get
\begin{equation}\label{eq:eigen-hatCp-Cp}
\bs{\Pi}_r \widehat{\bs{C}}_P \bs{v}_i = \nu_i \bs{\Pi}_r \bs{C}_P \bs{v}_i \quad \stackrel{\eqref{eq:proj-hatCp}}{\Longleftrightarrow} \quad \bs{\Pi}_r \bs{C}_P \bs{\Pi}_r^\top \bs{v}_i = \nu_i \bs{\Pi}_r \bs{C}_P \bs{v}_i.
\end{equation}
If~$\bs{v}_i \in \mathrm{Span}(\bs{\Pi}_r^\top)$, then~$\bs{\Pi}_r \bs{C}_P \bs{v}_i = \nu_i \bs{\Pi}_r \bs{C}_P \bs{v}_i$ and~$\nu_i = 1$.

Moreover, the approximate posterior covariance expression~\eqref{eq:approx-post-blg} yields
\begin{align}
\forall 1\leq i \leq n, \quad \widehat{\bs{C}}_P \bs{u}_i &= \bs{C}_{\prior} \bs{u}_i + \bs{\Pi}_r \bs{C}_P \bs{\Pi}_r^\top \bs{u}_i - \bs{\Pi}_r \bs{C}_{\prior} \bs{u}_i \notag\\
&\stackrel{\eqref{eq:gen-eigen}}{=}\lambda_i^{-1}\bs{C}_P \bs{u}_i + \bs{\Pi}_r \bs{C}_P \bs{\Pi}_r^\top \bs{u}_i - \bs{\Pi}_r \bs{C}_{\prior} \bs{u}_i.\label{eq:hatCp-ui}
\end{align}
If~$i>r$, then~$\bs{\Pi}_r^\top \bs{u}_i = 0$ and~$\bs{\Pi}_r \bs{C}_{\prior} \bs{u}_i = 0$, Eq.~\eqref{eq:hatCp-ui} simplifies to
\begin{equation}
    \widehat{\bs{C}}_P \bs{u}_i = \lambda_i^{-1}\bs{C}_P \bs{u}_i,
\end{equation}
stating that~$(\lambda_i^{-1}, \bs{u}_i)$ is an eigenpair of~$(\widehat{\bs{C}}_P, \bs{C}_P)$. 

Therefore, the~$n$ eigenpairs~$(\nu_i,\bs{v}_i)_{i=1}^n$ of~$(\widehat{\bs{C}}_P, \bs{C}_P)$ can be decomposed into two subparts: 
\begin{equation}
(\nu_i, \bs{v}_i) = \left\{ \begin{array}{l} 
    (1, \bs{v}_i), \bs{v}_i \in \mathrm{Span}(\bs{\Pi}_r^\top) \text{ if } i\leq r, \\
    (\lambda_i^{-1}, \bs{u}_i) \text{ otherwise}.
\end{array}\right.
\end{equation}

Finally, the Förstner distance between~$\widehat{\bs{C}}_P$ and~$\bs{C}_P$ is
\begin{equation}
d_\mc{F}(\widehat{\bs{C}}_P, \bs{C}_P) = \suml_{i=1}^n \log^2(\nu_i) = \suml_{i=r+1}^{n} \log^2 \left( \lambda_i^{-1}\right) = \suml_{i=r+1}^{n} \log^2 (\lambda_{i}).
\end{equation}
In the \gls{blg} case, for all~$1\leq i \leq n$,~$\lambda_i \leq 1$ the posterior variance is smaller than the prior one by law of total variance. The largest eigenvalues are therefore the closest to~$1$ and consequently the minimizer of the Förstner distance. Defining~$\mathrm{Ker}(\bs{\Pi}_r^\top) = \mathrm{Span}(\{\bs{u}_i\}_{r\leq i \leq n})$ is optimal, hence the projector construction. It is clear that, if~$r$ is not fixed, the projector is such that~$r=n$ and~$\bs{\Pi}_r = \mathrm{I_n}$. Note that we could consider a larger class of loss functions, as defined in~\cite[Definition~2.1]{spantini2015}.\\ 
\phantom{p} \hfill \cqfd
\item In the \gls{blg} case, the \gls{kld} is the sum of two terms: one depends only on the covariances and the other corresponds to a Mahalanobis distance. The first term is minimized using Theorem~\ref{theo:blg-proj}.
The second term is a Bregman divergence so that, using the same projector, the approximate posterior defined using~$\mc{L}_{\Pi_r}$ (as expressed in Definition~\ref{def:approx-likelihood}) is optimal according to~\cite{zahm2022}. 
Since both terms are individually optimal, combining the optimal likelihood approximation with the optimal projector yields a globally better posterior approximation than Spantini's original formulation.\hfill \cqfd
\end{enumerate}
\end{proof}
This approximate posterior in fact leads to the same Förstner distance 
value than the one obtained in~\cite[Theorem~2.3]{spantini2015}, so that the projector~$\bs{\Pi}_r$ leads to a minimizer of the Förstner distance for both likelihood approximations. The difference between the two formulations is the choice of the approximate likelihood: the approximation in Section~\ref{subsection:dr:approx-likelihood} is optimal given the projector~$\bs{\Pi}_r$.

%% file: section3.tex
\section{Gradient-free projector using covariance ratio}\label{section:cis}
In this section, we aim to derive a new gradient-free projector. In Section~\ref{subsection:cis:gf-projector}, we reformulate the previously defined projector without gradients, which appears to be an optimal projector in the \gls{blg} case and we extend it to the nonlinear case. Section~\ref{subsection:cis:numerical} focuses on the numerical approximation of the posterior covariance matrix, as well as theoretical bounds to assess the approximation quality. Finally, Section~\ref{subsection:cis:algo} presents the iterative algorithms applied in practice.

\subsection{Gradient-free projector derivation}\label{subsection:cis:gf-projector}
In our work, Projector~\eqref{eq:spantini-proj} is expressed using covariance-matrices information without requiring Hessian computation.
Indeed, as $\bs{W}_r = \bs{C}_\pi \bs{U}_r$ in the \gls{blg} case, an equivalent definition of $\bs{\Pi}_r$ can be stated as follows.
\begin{definition}[CIS projector]\label{def:cis-proj}
Let~$\bs{C}_\pi$ and~$\bs{C}_P$ denote the prior and posterior covariance matrices. Then, the \gls{cis} projector~$\bs{\Pi}_r$ is defined as 
\begin{equation}\label{eq:cis-proj}
\bs{\Pi}_r = \bs{C}_{\prior} \bs{U}_r \bs{U}_r^\top,
\end{equation}
where~$\bs{U}_r$ are the~$r$ eigenvectors of the pencil~$(\bs{C}_P,\bs{C}_\pi)$, associated with the smallest eigenvalues.
\end{definition}
Projector~\eqref{eq:cis-proj}, associated with the approximate likelihood of Definition~\ref{def:approx-likelihood}, is optimal in the linear case since it corresponds to the projector used in Theorem~\ref{theo:blg-proj}. The proposed \gls{dr} approach is informed only by the covariance matrices, since expression~\eqref{eq:cis-proj} relies on the eigenvectors of the pencil~$(\bs{C}_P,\bs{C}_{\prior})$. We name this approach \acrfull{cis} by analogy with the \acrfull{lis}~\cite{cui2014}.
The decomposition into two subspaces introduced in Eq.~\eqref{eq:red-form} rewrites
\begin{equation}\label{eq:red-form-cis}
\bx = \bx_r + \bx_\perp, \quad \text{with } \left\{ \begin{array}{l}
    \bx_r = \bs{C}_\pi \bs{U}_r \bs{z}_r = \bs{\Pi}_r \bx,\\
    \bx_\perp = \bs{C}_\pi \bs{U}_\perp \bs{z}_\perp = (\mathrm{I}_n - \bs{\Pi}_r)\bx,
\end{array}\right.
\end{equation}
given $\bs{U}_\perp = \left[ \bs{u}_{r+1}, \dots, \bs{u}_n \right] \in \R^{n \times (n-r)}$, so that the non-informed subspace is obtained from the completeness of the eigenvector basis.
In addition, as explained in~\cite[Equation~(3.5)]{spantini2015}, the \gls{cis} projector can be interpreted with the generalized Rayleigh quotient of the pencil $(\bs{C}_P,\bs{C}_\pi)$. Indeed, the eigenvectors~$\bs{u}_i$ correspond to the minimizer of the generalized Rayleigh quotient
\begin{equation}
\mc{R}(\bs{v}) = \frac{ \bs{v}^\top \bs{C}_P \bs{v}}{\bs{v}^\top \bs{C}_{\prior} \bs{v}} = \frac{ \mathbb{V}\mathrm{ar}_P(\bs{v}^\top X )}{\mathbb{V}\mathrm{ar}_\pi(\bs{v}^\top X)}
\end{equation}
over subspaces spanned by the eigenvectors associated to a larger eigenvalue. This ratio directly quantifies how much the posterior variance is reduced compared to the prior along direction~$\bs{v}$. A ratio close to 0 indicates a data-informed direction (highly concentrated posterior), while a ratio near 1 suggests a prior-dominated direction. 
The \gls{cis} projector therefore identifies the directions of maximum variance reduction, precisely the subspace where observations concentrate the posterior distribution the most. 

The nonlinear extension of the \gls{cis} projector is not equivalent to either Projector~\eqref{eq:spantini-proj} or to the gradient-based ones, because $\bs{W}_r \neq \bs{C}_\pi \bs{U}_r$ in the general case. However, building up on this interpretation, we could extend this approach to nonlinear problems provided that a direction is informed if and only if the posterior variance along this direction is reduced compared to the prior. Although this is a reasonable assumption in many cases, this is not always verified in the general Bayesian framework (see Appendix~\ref{appendix:reduc-variance}). Optimality results cannot be guaranteed in the nonlinear case, since the \gls{kld} and Hellinger distance depend on more than just the eigenvalues of~$(\bs{C}_P,\bs{C}_\pi)$ and approximating only the second order moment is theoretically insufficient.

\subsection{Numerical considerations}\label{subsection:cis:numerical}
We first specify the weighted estimator used to approximate the posterior covariance matrix in Section~\ref{subsection:cis:numerical-approxcov} and then focus on proposing bounds for the approximation quality in the nonlinear case in Section~\ref{subsection:cis:numerical-bounds}.

\subsubsection{Approximating the posterior covariance matrix}\label{subsection:cis:numerical-approxcov}
The \gls{cis} projector requires the estimation of the posterior covariance matrix $\bs{C}_P$ instead of the Hessian matrix $\bs{H}$ in the gradient-based approaches.
We address this using \gls{wmc} sampling,
\begin{align}
    \bs{C}_P &= \Int_{\R^n} (\bx - \E_P(\bx))(\bx - \E_P(\bx))^\top P(\bx)\ d\bx, \notag\\
    &= \Int_{\R^n} (\bx - \E_P(\bx))(\bx - \E_P(\bx))^\top \frac{P(\bx)}{P_{\Pi_r}(\bx)} P_{\Pi_r}(\bx)\ d\bx, \notag\\
    &= \Int_{\R^n} (\bx - \E_P(\bx))(\bx - \E_P(\bx))^\top \frac{\mc{L}(\bx;Y)}{\mc{L}_{\Pi_r}(\bs{z}_r;Y)} P_{\Pi_r}(\bx)\ d\bx,\\
    &\coloneqq \Int_{\R^n} (\bx - \E_P(\bx))(\bx - \E_P(\bx))^\top \omega(\bx) P_{\Pi_r}(\bx)\ d\bx. \label{eq:def-weights}
\end{align} 
The key to this estimation is recognizing that we can rewrite the posterior covariance matrix as a weighted expectation with respect to a posterior approximation~$P_{\Pi_r}$, using the weights $\omega(\bx)$. The mean is also written as a weighted quantity 
\begin{equation} 
    \E_P(\bx) = \Int_{\R^n} \bx P(\bx)\ d\bx  = \Int_{\R^n} \bx \omega(\bx) P_{\Pi_r}(\bx)\ d\bx = \E_{P_{\Pi_r}}(\bx \omega(\bx)).
\end{equation}
In that case, one can approximate~$\E_P(\bx)$ using~$n$ samples~$\bx^{(i)}$ drawn according to~$P_{\Pi_r}$: 
\begin{equation}
   \E_P(\bx) \simeq \widetilde{\bs{\mu}}_P = \suml_{i=1}^n \omega^{(i)}\bx^{(i)} / \left(\suml_{i=1}^n \omega^{(i)}\right),
\end{equation}
where the variance of the weights is an indicator of the estimator accuracy. The expectation of the weights along~$P_{\Pi_r}$ is~$1$. A simple estimator for~$\bs{C}_P$ would be, considering a Bessel's correction,
\begin{equation}\label{eq:estimator-Cp-bessel}
    \widetilde{\bs{C}}_P = \tfrac{\suml_{i=1}^n \omega^{(i)}}{\left(\suml_{i=1}^n \omega^{(i)} \right)^2 - \suml_{i=1}^n (\omega^{(i)})^2}\sum_{i=1}^n \omega^{(i)} (\bx^{(i)} - \widetilde{\bs{\mu}}_P)(\bx^{(i)} - \widetilde{\bs{\mu}}_P)^\top.
\end{equation} 

\subsubsection{Bounds for the approximation quality}\label{subsection:cis:numerical-bounds}
To assess the quality of our approximation in the nonlinear case, we rely on existing theoretical results. Specifically, the squared Hellinger distance and the numerical bound presented in~\cite{cui2022a} are rewritten using the \gls{wmc} weights and combined. Propositions~\ref{prop:hellinger-weights} and~\ref{prop:final-bound} summarize the main results.
\begin{proposition}[Hellinger distance using weights]\label{prop:hellinger-weights}
    Using the approximate posterior~\eqref{eq:post-approx-decomp}, the approximate likelihood of Definition~\ref{def:approx-likelihood} and the corresponding weights~\eqref{eq:def-weights}, the squared Hellinger distance between the posterior and its approximation is given by
    \begin{equation}
    d_H(P,P_{\Pi_r})^2 = 1 - \sqrt{1-\mathbb{V}\mathrm{ar}_{P_{\Pi_r}}(\sqrt{\omega(X)})}.
\end{equation}
\end{proposition}
\noindent The proof is available in Appendix~\ref{appendix:hellinger-weights}. With this expression, it is clear that finding an optimal approximation of the posterior is equivalent to minimizing the variance of the weights. In addition to the approximation error inherent in \gls{dr}, we must also account for the \gls{mc} error introduced by our finite sample estimation. Combining the previous results, we derive a bound that could be numerically useful to assess the quality of the approximation.
\begin{proposition}[Bound of the final approximation quality]\label{prop:final-bound}
The expected Hellinger distance between the posterior~$P$ and the numerically approximated posterior~$P_{\Pi_r}^{(N)}$ is upper bounded with
\begin{equation}
    \E_{MC} \left( d_H(P,P_{\Pi_r}^{(N)})^2 \right) \leq 1 - \sqrt{1-\mathbb{V}\mathrm{ar}_{P_{\Pi_r}}(\sqrt{\omega(X)})} + \frac{2}{N}\E_{P_{\Pi_r}}\left( \mathbb{V}\mathrm{ar}_{\pi_{\perp|r}}(\omega(X)) \right).
\end{equation}
\end{proposition}
\noindent The proof is available in Appendix~\ref{appendix:hellinger-bound}.
This bound can be used numerically to assess the convergence of our algorithm without any supplementary cost.

\subsection{Covariance-Informed Subspace (CIS) algorithm}\label{subsection:cis:algo}
Let us now focus on the method implementation. The construction relies on four algorithms serving distinct purposes, used sequentially depending on the problem difficulty. Algorithm~\ref{algo:iterative-cis} builds the informed subspace iteratively from weighted samples. It requires at each iteration samples from the approximate posterior, which are produced by Algorithm~\ref{algo:pseudo-marg-mean}. Once the subspace is fixed, Algorithm~\ref{algo:delayed} provides exact posterior sampling via a delayed acceptance scheme. Finally, when weight degeneracy is too severe for Algorithm~\ref{algo:iterative-cis} to initialize properly, Algorithm~\ref{algo:smc-cis} replaces it by coupling the subspace construction with an SMC tempering scheme. Sections~\ref{subsection:cis:algo-iterative} to~\ref{subsection:cis:algo-smc} detail each of these steps in turn.

\subsubsection{Iterative construction}\label{subsection:cis:algo-iterative}
We now present the iterative algorithm used to build the \gls{cis}.
As proposed by~\cite{cui2016a, zahm2022} for the gradient-based framework, we consider a sequence of posterior approximations~$\{ P^{(0)}, \dots, P^{(N_\mathrm{ite})}\}$, where~$P^{(0)} = \pi$ and~$P^{(i)} = P_{\Pi^{(i)}_{r}}$. This approach allows us to iteratively consider distributions that have a smaller discrepancy with respect to the posterior, hence diminishing weights variance, which improves the estimation of the posterior covariance matrix. The procedure is detailed in Algorithm~\ref{algo:iterative-cis}. 
Starting from $N_\mathrm{init}$ prior samples, $N_\mathrm{ite}$ steps are performed in order to estimate the informed subspace. At each iteration, the \gls{wmc} approximation of the posterior covariance matrix is computed using samples extracted from all the previous iterations in order to be more robust. Solving the eigenproblem associated with the covariance-matrices pencil, we deduce the associated projector $\bs{\Pi}_r^{(i)}$ and its rank. The projector rank~$r$ is also varying along iterations and should be denoted~$r^{(i)}$, however, it is dropped in this notation for the sake of simplicity. The rank is determined by identifying a plateau in the eigenvalues, assuming that this plateau is achieved when the corresponding directions become uninformative. Specifically, after identifying the index for which the relative variation between two eigenvalues reaches its maximum, we choose the rank of the projector to be the last index following this maximum for which the variation is greater than~$10\%$. If we are not sure about the existence of such plateau, we choose $r$ such that the corresponding eigenvalue is less than a user-defined threshold. This latter choice exploits the interpretation of the \gls{cis} problem eigenvalues: an eigenvalue~$\lambda$ indicates that the posterior variance is $\lambda$ times the prior variance along the associated eigendirection. Moreover, in order to avoid numerical degeneracy in the first steps of iterations, we also set evolving bounds for the number of informed dimensions. The sampling of the approximate reduced posterior is detailed in the next section. Since this algorithm requires samples from the full approximate posterior, we then draw samples from the non-informed subspace conditionally to some reduced samples extracted from the chain.
\begin{algorithm}[h]
    \caption{CIS construction}\label{algo:iterative-cis}
    \begin{algorithmic}[1]
    \Require~$\pi_X$,~$N_\mathrm{init}$,~$N_\mathrm{ite}$
    \State \textbf{Initialization: } Sample~$\{\bx^{(k)}\}_{k=1}^{N_\mathrm{init}} \sim \pi_X$
    \For{$1\leq i \leq N_\mathrm{ite}$}
        \State Estimate~$\widetilde{\bs{C}}_P^{(i)}$ using \gls{wmc} with previous samples (Eq.~\eqref{eq:estimator-Cp-bessel})\label{algo:line:hatCp}
        \State Compute the matrix~$\bs{U}^{(i)}$ containing the eigenvectors of the pencil~$(\widetilde{\bs{C}}_P^{(i)}, \bs{C}_{\prior})$
        \State Choose~$r$, deduce~$\bs{\Pi}_r^{(i)} = \bs{C}_\pi \bs{U}_{r}^{(i)} \bs{U}_{r}^{(i),\top}$, $\bs{V}_r^{(i)}=\bs{C}_\pi \bs{U}_r^{(i)}$ and its orthogonal.
        \State Sample~$P_{\Pi_r^{(i)}}: \bs{z}_r^{(k)} \sim P_{\Pi_r^{(i)},r}, \ \bs{z}^{(k)}_\perp \sim \pi_{\perp|r}(\cdot|\bs{z}_r^{(k)})$\label{algo:line:sample}
        \State Add approximate full posterior samples~$\bx^{(k)} = \bs{V}_r^{(i)} \bs{z}_r^{(k)} + \bs{V}_\perp^{(i)} \bs{z}_\perp^{(k)}$ to the set of samples.
    \EndFor
    \State \Return~$\bs{V}_r$, $\bs{V}_\perp$
    \end{algorithmic}
\end{algorithm}
The total number of iterations $N_\mathrm{ite}$ is assessed by verifying whether the subspace construction has effectively converged. To compare subspaces, we measure the distance between them using principal angles~\cite{bjorck1973}, which generalize the angle between vectors to subspaces and can be computed via the SVD of the product of orthonormal bases. We also estimate the quantities involved in the bound of Proposition~\ref{prop:final-bound}, that is to say $\mathbb{V}\mathrm{ar}_{P_{\Pi_r}}(\sqrt{\omega})$ (or equivalently $\E_{P_{\Pi_r}}(\sqrt{\omega})$) and $\E_{\post[\Pi_r]} \left( \mathbb{V}\mathrm{ar}_{\prior[\perp|r]}(\omega(X)) \right)$, using a \gls{mc} approximation with the current samples.

\subsubsection{Approximate posterior sampling algorithm}\label{subsection:cis:algo-sampling}
We now detail the procedure to sample the approximate marginal posterior on the informed subspace~\cite{cui2021a} required in Algorithm~\ref{algo:iterative-cis} (line~\ref{algo:line:sample}). As designed in the previous algorithm, the general principle of sampling an approximate \gls{dr} posterior is to sample the posterior on the informed subspace only while sampling the prior on the non-informed subspace.
Algorithm~\ref{algo:pseudo-marg-mean} is a \gls{mh} algorithm on the informed subspace.
In theory, the approximate likelihood depends on \gls{mc} prior samples on the non-informed subspace as expressed in~\eqref{eq:approx-likelihood-mc}.
However, the projector at the first steps may not be able to correctly describe the informed subspace of interest. As a consequence, the \gls{mc} estimator of the likelihood exhibits a large variance. If we increase the sample size, the variance is reduced but this option is computationally expensive, whereas if we keep a small sample size, the chain convergence cannot be achieved. Indeed, the likelihood estimator can be viewed as a stochastic quantity and the ratio between the observation noise and the estimator variance determines whether or not the algorithm could converge, in particular, whether the accept rate drops to zero~\cite{lovbak2026}. In the first iterations with a coarse approximation of the projector, the introduced \gls{mc} noise is indeed much larger than the observation noise, leading to sampling difficulties. In practice, we discard the estimator variability by using the deterministic prior mean of~$Z_\perp$ conditionally to~$\bs{z}_r$ as the unique \gls{mc} sample, as stated in Algorithm~\ref{algo:pseudo-marg-mean}. The deterministic mean is computationally tractable and does not require any numerical approximation. 
\begin{algorithm}[h]
    \caption{Pseudo-marginal sampling with deterministic prior mean}\label{algo:pseudo-marg-mean}
    \begin{algorithmic}[1]
    \Require $\bs{\Pi}_r$, $N_s$, $\pi_\mathrm{prop}(\bs{z}_r, \cdot)$
    \State \textbf{Initialization: }~$\bs{z}_r^{(0)} \sim \prior[r]$, $\bs{z}_\perp^{(0)} = \E_{\prior[\perp|r]}(Z_\perp|\bs{z}_r^{(0)})$
    \For{$1\leq k \leq N_s$}
        \State Draw a proposal~$\bs{z}_r' \sim \pi_\mathrm{prop}(\bs{z}_r^{(k-1)}, \cdot)$, set~$\bs{z}_\perp' = \E_{\prior[\perp|r]}(Z_\perp|\bs{z}_r')$
        \State Compute the acceptance probability 
         \begin{equation}\label{eq:MHratio}
            \widehat{r}_\mathrm{MH}(\bs{z}_r'|\bs{z}_r^{(k-1)}) = \min \left\{ 1, \frac{\post(\bs{z}_r',\bs{z}_\perp')\pi_\mathrm{prop}(\bs{z}_r^{(k-1)}, \bs{z}_r')}
            {\post(\bs{z}_r^{(k-1)},\bs{z}_\perp^{(k-1)})\pi_\mathrm{prop}(\bs{z}_r', \bs{z}_r^{(k-1)})}
            \right\}.
        \end{equation}
        \setlength{\abovedisplayshortskip}{0pt}
        \setlength{\abovedisplayskip}{0pt}
        \State
        Set $
        \left(\bs{z}_r^{(k)}, \bs{z}_\perp^{(k)}\right) = \left\{ 
            \begin{array}{l}\left(\bs{z}_r',\bs{z}_\perp'\right) \text{with probability~}\widehat{r}_\mathrm{MH}(\bs{z}_r'|\bs{z}_r^{(k-1)}),\\
         \left(\bs{z}_r^{(k-1)},\bs{z}_\perp^{(k-1)}\right) \text{ otherwise}.\end{array}\right.
        $
    \EndFor
    \State \Return the \gls{mcmc} chain~$\{\bs{z}_r^{(k)}\}_{k=1}^K$
    \end{algorithmic}
\end{algorithm}

\subsubsection{Exact posterior sampling}\label{subsection:cis:algo-final}
Once the projector is built, the exact posterior can be sampled using a delayed algorithm detailed in Algorithm~\ref{algo:delayed}~\cite{cui2021a}. The algorithm is based on two nested accept-reject stages. The first stage corresponds to the one of Algorithm~\ref{algo:pseudo-marg-mean} and proposes a new state in the informed subspace. If the proposal is accepted, a new state in the non-informed subspace is proposed in the second stage. When the \gls{dr} correctly describes the informed subspace, using the mean approximation of the likelihood is appropriate in the first stage. The \gls{mh} ratio~\eqref{eq:mhratio-beta} in the second stage reflects the quality of the projector construction, with values close to one indicating a relevant projector.
\begin{algorithm}[h]
    \caption{Exact posterior sampling using a delayed algorithm}\label{algo:delayed}
    \begin{algorithmic}[1]
    \Require $\bs{V}_r$, $\bs{V}_\perp$, $N_s$, $\pi_\mathrm{prop}(\bs{z}_r,\cdot)$
    \State \textbf{Initialization: }~$\bx^{(0)} = \bs{V}_r \bs{z}_r^{(0)} + \bs{V}_\perp \bs{z}_\perp^{(0)} \sim \prior$
    \For{$1\leq k \leq N_s$}
        \State Draw a proposal~$\bs{z}_r' \sim \pi_\mathrm{prop}(\bs{z}_r^{(k-1)}, \cdot)$
        \State Compute the acceptance probability~\eqref{eq:MHratio} using the mean approximation
        \State With probability~$\widehat{r}_\mathrm{MH}(\bs{z}_r'|\bs{z}_r)$:
        \State \hspace{1cm} Draw a proposal for the non-informed parameter~$\bs{z}_\perp' \sim \prior[\perp|r](\cdot | \bs{z}_r')$
        \State \hspace{1cm} Compute the likelihood~$\mc{L}(\bs{V}_r \bs{z}_r' + \bs{V}_\perp \bs{z}_\perp')$
        \State \hspace{1cm} Compute the second-stage acceptance probability 
        \begin{equation}\label{eq:mhratio-beta}
            \beta(\bs{z}_r',\bs{z}_\perp' | \bs{z}_r^{(k)}, \bs{z}_\perp^{(k)}) = \frac{\mc{L}(\bx')\mc{L}_{\Pi_r}(\bs{z}_r^{(k)})}{\mc{L}(\bx^{(k)})\mc{L}_{\Pi_r}(\bs{z}_r')}.
        \end{equation}
        \State \hspace{1cm} Set $
        \left(\bs{z}_r^{(k)}, \bs{z}_\perp^{(k)}\right) = \left\{ 
            \begin{array}{l}\left(\bs{z}_r',\bs{z}_\perp'\right) \text{with probability~}\beta(\bs{z}_r'\bs{z}_\perp' | \bs{z}_r^{(k)}, \bs{z}_\perp^{(k)}),\\
         \left(\bs{z}_r^{(k-1)},\bs{z}_\perp^{(k-1)}\right) \text{ otherwise}.\end{array}\right.
        $      
        \State Otherwise, reject:~$(\bs{z}_r^{(k)},\bs{z}_\perp^{(k)}) = (\bs{z}_r^{(k-1)},\bs{z}_\perp^{(k-1)})$  
    \EndFor
    \State \Return the \gls{mcmc} chain~$\{\bs{z}_r^{(k)}, \bs{z}_\perp^{(k)}\}_{k=1}^K$
    \end{algorithmic}
\end{algorithm}

\subsubsection{Dealing with weight degeneracy}\label{subsection:cis:algo-smc}
If the posterior variance is really reduced relative to the prior distribution, only a small number of samples carry all the weights of the covariance matrix estimation. In~\cite{zahm2022}, weights are set to one at the initialization step to avoid degeneracy due to poor approximation. This option cannot be combined with the \gls{cis} approach, as it would lead to an uninformative eigenproblem linked to the pencil~$(\widetilde{\bs{C}}_{\prior},\bs{C}_{\prior})$. Some state-of-the-art methods involving weights modifications to mitigate the weights degeneracy in the context of importance sampling for rare event estimation are presented in~\cite{elmasri2022}. In this work, we integrate the \gls{cis} approach within a \gls{smc} framework to overcome the weight degeneracy. The procedure, mainly based on the implementation detailed in~\cite{carrera2024}, is given in Algorithm~\ref{algo:smc-cis}.
\begin{algorithm}
    \caption{CIS construction using SMC}\label{algo:smc-cis}
    \begin{algorithmic}[1]
    \Require~$\prior$,~$\mc{L}$,~$N_\mathrm{samp}$,~$N_\mathrm{ite}$,~$N_\perp$
    \State \textbf{Initialisation: }~$t = 0$,~$\beta^{(t)} = 0$,~$\bs{X}^{(0)} = \{\bx^{(0,i)}\}_{i=1}^{N_\mathrm{samp}} \sim \pi$ 
    \State Compute~$\bs{L}^{(0)} = \{\mc{L}(\bx^{(0,i)})\}_{i=1}^{N_\mathrm{samp}}$, set~$\bs{L}_{\Pi_r}^{(0)} = \{ \ol{\bs{L}}^{(0)}\}_{i=1}^{N_\mathrm{samp}}$
    \State~$\beta^{(1)} \leftarrow$ \texttt{Update$_\beta$}($\beta^{(0)}$,~$\bs{L}^{(0)}$,~$\bs{L}_{\Pi_r}^{(0)}$) \Comment{Using an ESS target value}
    \While{$\beta^{(t+1)} < 1$}
        \State~$t \leftarrow t+1$
        \State For $0\leq k <t$, compute the weight vector~$\bs{\omega}^{(t,k)} \in \R^{N_\mathrm{samp}}$ using~$\left(\bs{L}^{(k)}\right)^{\beta^{(t)}} $ and~$\left(\bs{L}_{\Pi_r}^{(k)}\right)^{\beta^{(k)}}$
        \State Estimate~$\widetilde{\bs{C}}_P^{(t)}$ using~$\{\bs{X}^{(k)}\}_{k=0}^{t-1}$ and~$\{\bs{\omega}^{(t,k)}\}_{k=0}^{t-1}$
        \State Solve CIS eigenproblem~$(\widetilde{\bs{C}}_P^{(t)},\bs{C}_{\prior})$ to obtain the eigenvector matrix~$\bs{U}^{(t)}$
        \State Deduce~$r$ and construct~$\bs{\Pi}_r^{(t)} = \bs{C}_\pi \bs{U}_r^{(t)} \bs{U}_r^{(t),\top}$, $\bs{V}_r^{(t)} = \bs{C}_\pi \bs{U}_r^{(t)}$
        \State $\bs{X}^{(t)} \leftarrow$ Redraw of~$\bs{X}^{(t-1)}$ with weights~$\bs{\omega}^{(t,t-1)}$
        \State $\bs{Z}_r \leftarrow$~$N_\mathrm{ite}$ random moves with initial state~$\bs{Z}_r = \bs{U}_r^{(t),\top} \bs{X}^{(t)}$ and proposal kernel~$\widetilde{\mathbb{C}\mathrm{o}}\mathrm{v}(\bs{Z}_r)$
        \State For each particle $\bs{Z}_r$, draw~$N_\perp$ samples~$\bs{Z}_\perp = \{\bs{z}_\perp^{(k)}\}_{k=1}^{N_\perp} \sim \pi_{\perp|r}$
        \State~$\bs{X}^{(t)} \leftarrow \bs{V}_r^{(t)} \bs{Z}_r + \bs{V}_\perp^{(t)} \bs{Z}_\perp$
        \State Compute $\bs{L}^{(t)}$, $\bs{L}_{\Pi_r}^{(t)}$
        \State~$\beta^{(t+1)} \leftarrow$ \texttt{Update$_\beta$}($\beta^{(t)}$,~$\bs{L}^{(t)}$,~$\bs{L}_{\Pi_r}^{(t)}$) 
    \EndWhile
    \State \Return~$\bs{\Pi}_r^{(t)}$
    \end{algorithmic}
\end{algorithm}
At each \gls{smc} iteration~$t$, the \gls{mc} particles move is performed on the current reduced subspace defined by~$\bs{\Pi}_r^{(t)}$ with a \gls{smc} tempering parameter~$\beta^{(t)}$. Then, particles on the full space are obtained by drawing prior conditional samples in the non-informed subspace. In order to improve the computational efficiency, the posterior covariance matrix approximation defined in Eq.~\eqref{eq:estimator-Cp-bessel} uses samples from all previous iterations, so that we need to save their corresponding likelihoods and tempering parameter, denoted~$\bs{L}$ and~$\beta$ in the algorithm.

The \gls{smc} redrawing of the particles uses the current sample weights with a possible clipping and tempering regularization~\cite[Chapter~2.2.1]{elmasri2022}. The resulting samples are then moved using a Markov kernel in order to avoid redundancy in the sample set. The tempering parameter update requires to solve the minimization problem as done in~\cite{beskos2016,carrera2024},
\begin{equation}
\beta^{(t)} = \argmin\limits_{\beta \in (\beta^{(t-1)},1]} \left( \mathrm{ESS}(\beta) - \tau\right)^2,\text{ with }\mathrm{ESS}(\beta) = \frac{\left(\sum_{i=1}^{N_\mathrm{samp}} \omega(\bx^{(t-1,i)},\beta)\right)^2}{\sum_{i=1}^{N_\mathrm{samp}} \omega(\bx^{(t-1,i)},\beta)^2},
\end{equation}
where $\mathrm{ESS}$ is the effective sample size.
In practice, the target value~$\tau$ belongs to $(0,N_\mathrm{samp})$ and is set as a percentage of the number of samples, with smaller percentage used for more degenerate weight distributions.

To conclude this section, the \gls{cis} approach combines two key ingredients. The gradient-free Projector~\eqref{eq:cis-proj} is built based on posterior-to-prior covariance matrices ratios. The iterative estimation scheme detailed in Algorithm~\ref{algo:iterative-cis} allows progressively refining the subspace using \gls{wmc}. To handle posterior degeneracy in challenging cases, we also propose a \gls{cissmc} procedure detailed in Algorithm~\ref{algo:smc-cis}. In the next sections, we apply the \gls{cis} and the \gls{cissmc} algorithms and assess their computational performance on two configurations: a groundwater problem and a realistic atmospheric inverse problem exhibiting a strong posterior concentration.

%% file: section4.tex
\section{Application to a groundwater flow problem}\label{section:gw}
In this section, the \gls{cis} approach is applied to a groundwater flow inverse problem introduced in~\cite{constantine2016}. The problem setup and computational framework is established in Section~\ref{subsection:gw:setup}. Then, a comparison between the \gls{cis} construction and a gradient-based one, denoted \gls{gis} hereafter, is led in Section~\ref{subsection:gw:prior}. In Section~\ref{subsection:gw:iterative}, the iterative algorithm efficiency is assessed. Finally in Section~\ref{subsection:gw:efficiency} the approximation quality is estimated.

\subsection{Problem setup and model description}\label{subsection:gw:setup}
 The~$2$D groundwater flow problem is
 \begin{equation}
     \left\{
 \begin{alignedat}{3}\label{eq:gw:forward-pb}
     -\nabla \cdot (f(\bs{s}) \nabla h(\bs{s})) &= g, &&\quad \bs{s} \in \Omega=[0,1]^2, \\
     \nabla h(\bs{s}) \cdot \bs{n} &= 0, &&\quad \bs{s} \in \partial \Omega_N, \\
     h(\bs{s}) &= 0 &&\quad \bs{s} \in \partial \Omega_D, 
 \end{alignedat}
     \right.
 \end{equation}
 where~$h$ is the hydraulic head and~$f$ is the permeability field. A homogeneous Dirichlet condition is applied on~$\partial \Omega_D$ (left, top and bottom boundaries), while a homogeneous Neumann condition is applied on~$\partial \Omega_N$ (right boundary). The right-hand side term~$g$ is set to~$1$. Problem~\eqref{eq:gw:forward-pb} is solved using the~$\mathbb{P}_2$-triangular finite element method on a mesh generated using Triangle~\cite{shewchuk1996}.

 The goal of the inverse problem is to infer the spatially-varying permeability field~$f$ from hydraulic head measurements, a classic scenario in subsurface characterization. The prior of~$f$ follows a log-normal distribution with an isotropic exponential autocovariance function,
 \begin{equation}\label{eq:gw:field-param}
 \log f \sim \mc{N}(0,k) \quad \text{with} \quad \begin{aligned}
     &k(\bs{s},\bs{s}') = \exp\left(-\frac{\norm{\bs{s}-\bs{s}'}_{\ell_1}}{l}\right), \quad \forall \ \bs{s}, \bs{s}' \in \Omega, \\
 \end{aligned}
 \end{equation}
 and is parametrized using the \gls{kl} decomposition,
 \begin{equation}
 \log f(\bs{s},\bx) \simeq \suml_{i=1}^n \sqrt{\lambda_i}\phi_i(\bs{s})x_i.
 \end{equation}
 The correlation length~$l$ is set to~$0.02$, which is small relative to the size of the spatial domain.
 The true field~$f_\mathrm{true}$ used to generate the observations is a particular realization of the truncated \gls{kl} decomposition of the prior using $n=200$ \gls{kl} modes. This realization is shown in the top left plot of Fig.~\ref{fig:gw:priorreal}. The associated hydraulic head field~$h_\mathrm{true}$ (bottom left plot) is computed using a fine uniform triangulation with~$15,859$ elements. The observations $\by$ correspond to the hydraulic head~$h_\mathrm{true}$ measured on the right border at~$7$ sensor locations~$\bs{s}_R = (0, \{0.2,0.3,0.4,0.5,0.6,0.7,0.8\})$ (see Fig.~\ref{fig:gw:priorreal}, bottom left). Measurement errors are simulated by adding Gaussian noise with a standard deviation proportional to the average magnitude of the true value:~$\by_\mathrm{obs,i} = h_\mathrm{true}(\bs{s}_{R,i}) + \varepsilon_i$, with $\varepsilon_i \sim \mc{N}(0,(0.01 \times h_\mathrm{true}(\bs{s}_{R,i})^2)$.
 The \gls{kl} decomposition used for the inference is truncated to~$n=100$ modes so that the input parameter set~$\bx \in \R^n$ corresponds to the $100$ first \gls{kl} coordinates. Three prior realizations and their corresponding solutions are shown in Fig.~\ref{fig:gw:priorreal}. For the inference, the forward problem~\eqref{eq:gw:forward-pb} is solved on a coarser uniform triangulation with~$1,577$ elements and we model the noise by a centered additive i.i.d. Gaussian variable with standard deviation equal to $0.002$, which corresponds to $1\%$ of the average values magnitude.
 \begin{figure}
     \centering
     \includegraphics[width=0.85\textwidth, trim={0 6cm 0 0cm}, clip]{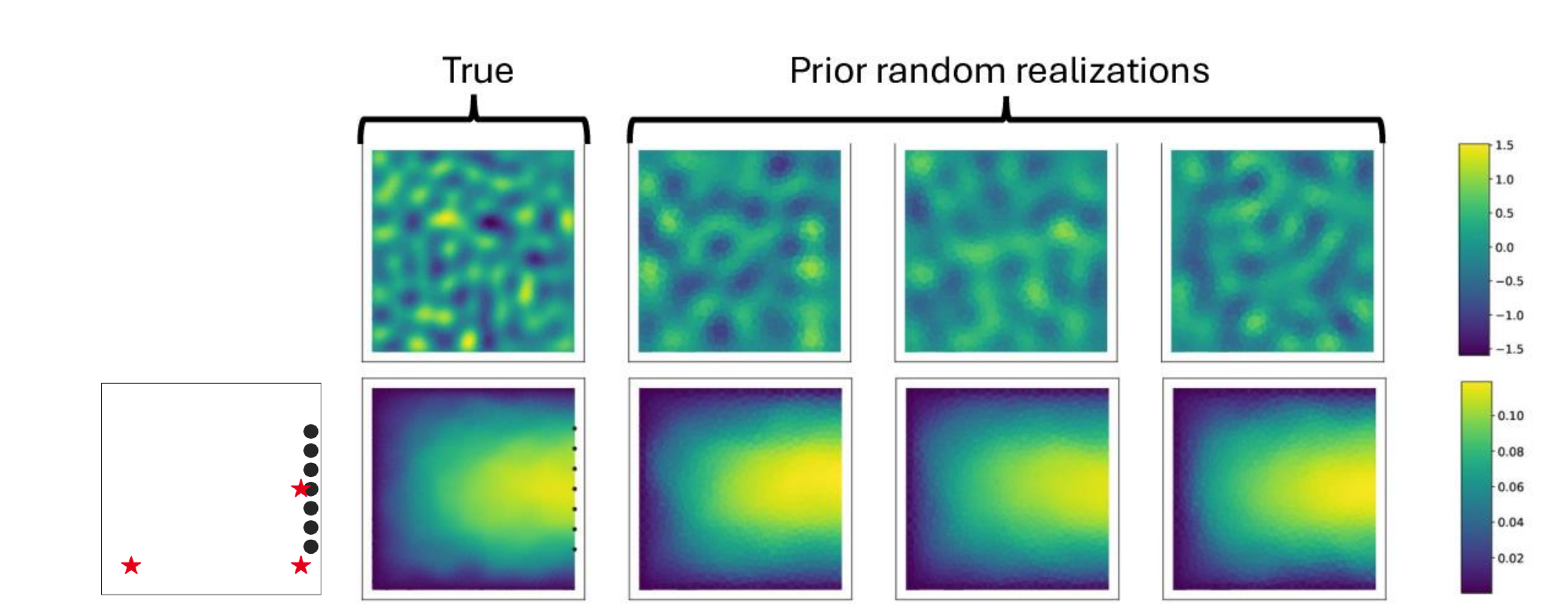}
     \caption{Groundwater case - (top) Few samples of~$\log f$ and (bottom) their corresponding solutions. (left) is the true log-field~($\log f_\mathrm{true}$) and its solution~($h_\mathrm{true}$). (bottom left) Black circles are sensor locations and red stars are selected spatial locations.}\label{fig:gw:priorreal}
 \end{figure}
 The sparse sensors configuration, where a few observations are  available only along the right boundary, creates an inherent asymmetry in the information content. We expect the \gls{dr} to reflect this geometric configuration, with the informed subspace capturing variations near the observed boundary.

\subsection{Weighted prior approximations}\label{subsection:gw:prior}
 In order to compare the \gls{cis} and \gls{gis} approaches, we begin by drawing~$N$ prior samples~$\{\bx^{(i)}\}_{i=1}^N$ and compute the log-likelihood and its gradient for each sample. These gradients are obtained by solving the adjoint problem. We then compute the \gls{wmc} approximation of the Fisher information matrix~$\bs{H}$~\cite{zahm2022},
 \begin{equation}\label{eq:approx-Hpost}
 \bs{H} \simeq \widetilde{\bs{H}}_P^{(N)} = \tfrac{1}{\suml_{i=1}^N \omega(\bx^{(i)})}\suml_{i=1}^N \omega(\bx^{(i)}) \nabla \log \mc{L}(\bx^{(i)}) \left(\nabla \log \mc{L}(\bx^{(i)})\right)^\top,
 \end{equation}
 and the \gls{wmc} approximation of the posterior covariance $\bs{C}_P$~\eqref{eq:estimator-Cp-bessel}.
 A draw according to the prior is equivalent to consider that the whole space is non-informed ($r=0$), so that the weights $\omega$, defined in~\eqref{eq:def-weights}, are just the likelihoods divided by the empirical mean over the~$N$ samples: $\mc{L}_{\Pi_r}^{(N)}(\bx^{(i)}) = \sum_{i=1}^N \mc{L}(\bx^{(i)})/ N$.
 
 \begin{figure}[b]
 \centering
 \includegraphics[height=0.2\textheight]{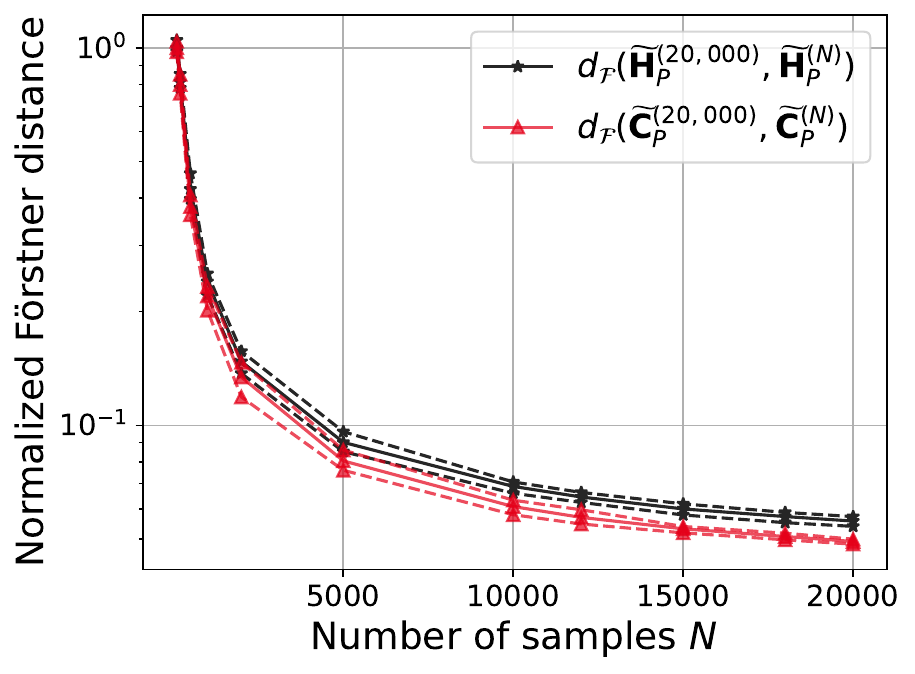}
 \hfill
 \includegraphics[height=0.2\textheight]{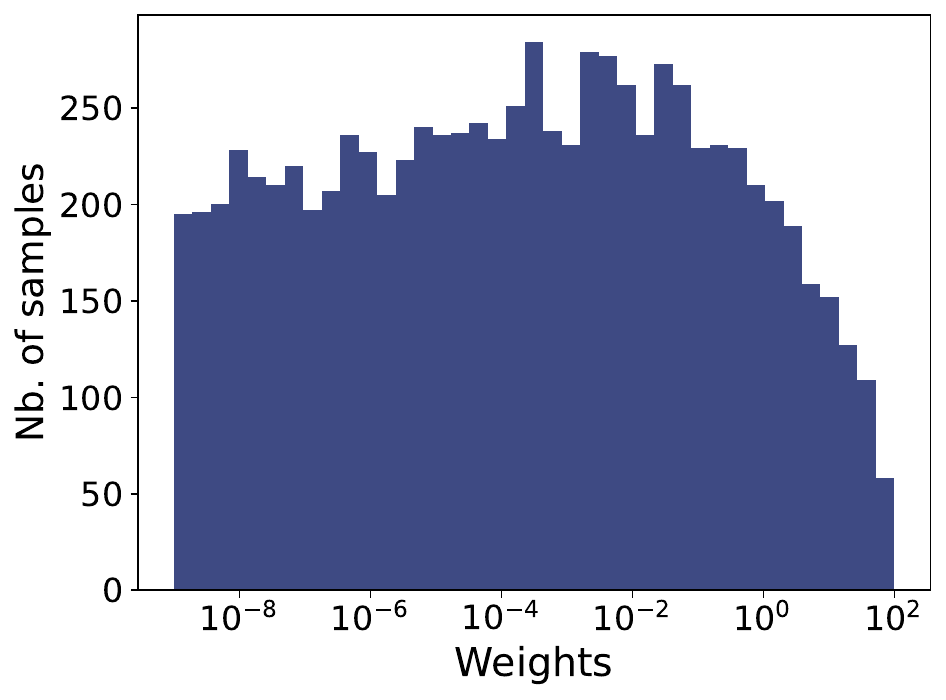}
     \caption{Groundwater case - (left) Convergence of the Förstner distance for the estimation of $\bs{H}_P$ (black stars) and $\bs{C}_P$ (red triangles). Plain line: mean value, dotted lines: min and max values using~$10 \times 20,000$ independent draws. (right) Distribution of the weights~$\omega$.
     }\label{fig:gw:convergence}
 \end{figure}
 We first assess the convergence of the \gls{wmc} approximations with respect to the number $N$ of prior samples. In Fig.~\ref{fig:gw:convergence} (left), we show the Förstner distance between the approximations and a reference approximation computed with an independent sample set of size $20,000$. The Förstner distance is normalized by the mean value obtained for $N=100$ for the sake of comparison. We observe a strong decrease until $N=5,000$ and smaller variations after. Convergence is similar for both $\bs{H}$ and $\bs{C}_P$. The weight distribution is also represented in Fig.~\ref{fig:gw:convergence} (right). Around 5\% of the samples have a weight greater than one, meaning that~$10,000$ samples correspond to~$500$ effective samples. This result is in agreement with the rule of thumb stating that around~$5 \times n$ samples are sufficient to approximate a $n$-dimensional matrix, with $n=100$ in our case. 
 Then, the two informed subspaces obtained from \gls{gis} and \gls{cis} are built by solving, respectively, the eigenproblems associated to~$(\widetilde{\bs{H}}_P, \bs{C}_\pi^{-1} = \mathrm{I}_n)$ and~$(\widetilde{\bs{C}}_P, \bs{C}_\pi = \mathrm{I}_n)$. 

 Figure~\ref{fig:gw:eigval}~(left) shows the eigenvalues associated with both gradient-based and gradient-free approaches. The eigenvalues of \gls{gis} decrease while the eigenvalues of \gls{cis} increase since the smallest eigenvalues are of interest in that method. In both approaches, a change in slope is observed after the fourth component, suggesting a four-dimensional informed subspace.
 \begin{figure}
 \centering 
 \includegraphics[height=0.2\textheight]{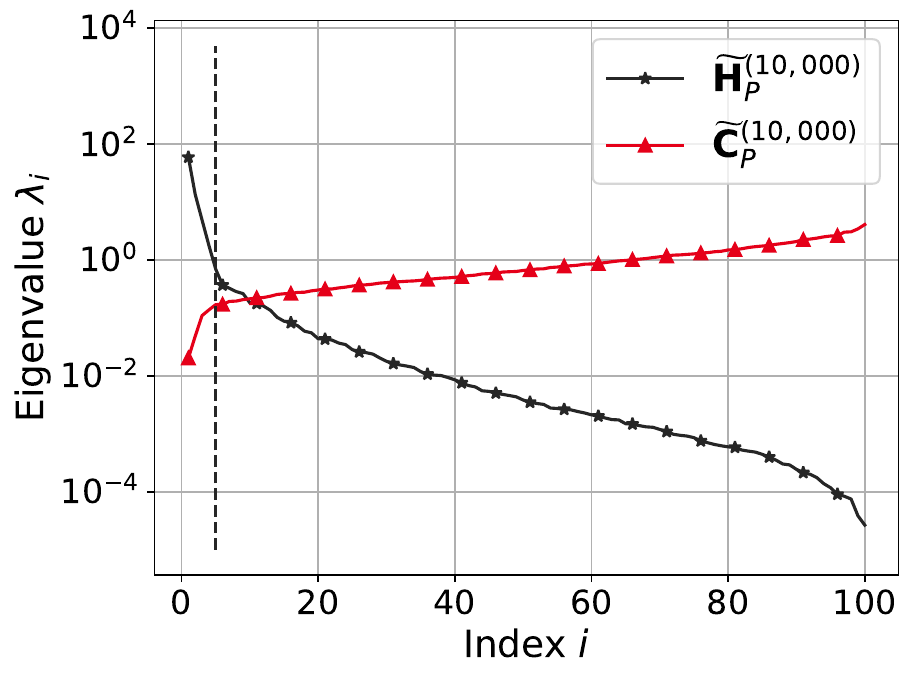}
 \hfill
 \includegraphics[height=0.2\textheight]{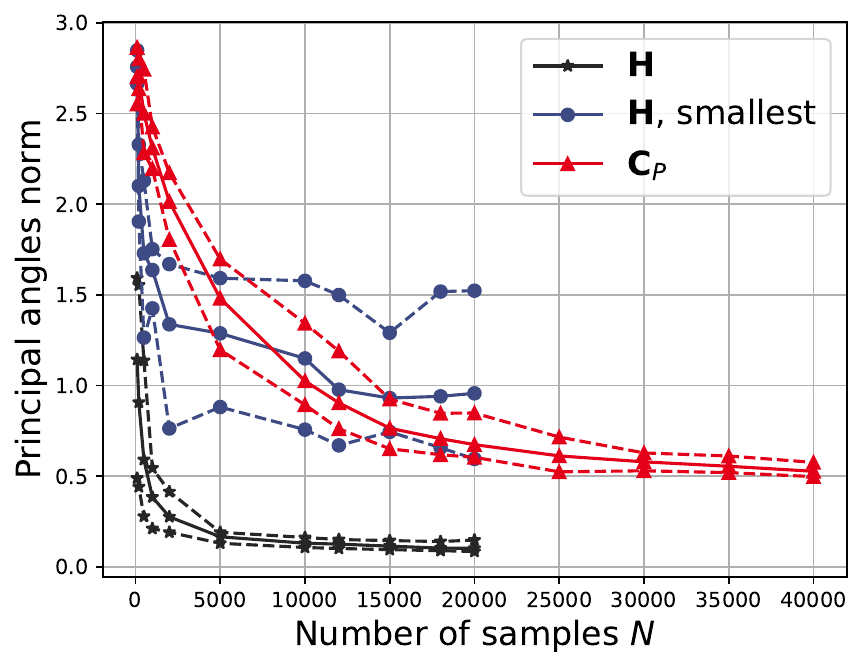}
 \caption{Groundwater case - (left) Eigenvalues obtained solving the gradient-based (black stars) and gradient-free (red triangles) eigenproblem using $10,000$ weighted prior samples. (right) 
 Convergence of the principal angles norm between a $4$-dimensional reference subspace and an estimated one according to the number of prior samples with $\bs{H}_P$ (black stars) and $\bs{C}_P$ (red triangles). Plain line: mean value, dotted lines: min and max values using~$10 \times 20,000$ (for $\bs{H}_P$) or $40,000$ (for $\bs{C}_P$) independent draws. For sake of comparison, subspace using the eigenvectors of $\bs{H}_P$ associated with the smallest eigenvalues (blue circles).}\label{fig:gw:eigval}
 \end{figure}
 To further investigate the convergence analysis, Figure~\ref{fig:gw:eigval}~(right) depicts the principal angle norm between a reference subspace ($r=4$) and subspaces computed using smaller sampling sizes. The slower convergence observed for the \gls{cis} can be attributed to its construction which relies on the smallest eigenvalues. To support this explanation, we also compute the smallest eigenvalues and the associated eigenspace for the gradient-based problem (blue circles) and observe a similarly slow convergence. This convergence analysis based on prior samples shows that the \gls{cis} method requires about $5$ times more samples than the \gls{gis} method to identify the reduced subspace ($25,000$ against $5,000$). However, the gradient-based approach requires computing both the likelihood and its gradient for each sample, whereas the gradient-free approach only uses likelihood evaluations. The cost of estimating the gradient for $K$ samples is roughly $2K$ forward model evaluations using the adjoint method (as done here), or $K\times n$ evaluations using finite differences. The \gls{cis} method is particularly advantageous in the latter case and can be further improved with the iterative algorithm, whose results are detailed in the next section.

We now investigate the log-permeability field corresponding to each reduced coordinate. Fig.~\ref{fig:gw:eigenspace} shows the fields obtained by successively inserting the first four reduced parameters into the \gls{kl} decomposition. The resulting fields illustrate the spatial perturbations of the log-permeability which, on average, have the greatest impact on the log-likelihood. Only the fields obtained with the \gls{cis} method are depicted, since those computed from \gls{gis} are similar.
\begin{figure}[!ht]
    \centering
    \includegraphics[width=0.9\textwidth]{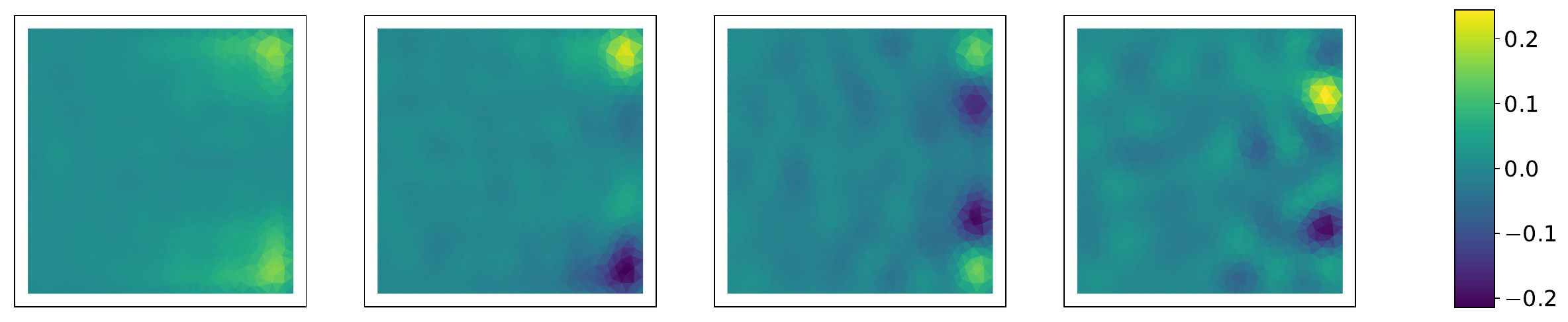} 
    \caption{Groundwater case - Log-coefficient fields obtained using the fourth first eigenvectors of \gls{cis} with $25,000$ prior samples.}\label{fig:gw:eigenspace}
\end{figure}
The main variations are concentrated near the right boundary where the sensors are located, consistent with evidence reported in previous works~\cite{constantine2016}.

\subsection{Iterative CIS results}\label{subsection:gw:iterative} 
We now apply the iterative \gls{cis} algorithm~\ref{algo:iterative-cis} and examine its convergence. The objective is to assess whether the algorithm remains robust with a small number of prior samples ($N_\mathrm{init}$ is set to $100$), and whether it can reduce the number of samples required to approximate the eigenspace. The rank is here determined by detecting a plateau in the eigenvalues, as explained in Section~\ref{subsection:cis:algo-iterative}. Moreover, to prevent the algorithm from selecting a high-dimensional subspace during the first iterations, a maximum dimension $r_\mathrm{max}$ is imposed: it is set to $5$ for the first iteration and then updated to $\mathrm{min}(r+2,10)$ to constrain subspace growth. To further reduce computational costs, we approximate the posterior covariance matrix at iteration~$K$  using samples from all previous steps. The sampling in the reduced subspace uses a short adaptive random walk Metropolis Hastings algorithm~\cite{haario2001,andrieu2008}. After this sampling, at most $80$ reduced posterior samples are extracted from the chain and $3$ samples in the non-informed subspace are drawn for each of them. Doing so, we obtain at most $240$ samples from the approximate posterior distribution. This algorithm is run several times to assess its robustness. 

The convergence of the algorithm is assessed using different indicators represented in Fig.~\ref{fig:gw:iterative}. On Fig.~\ref{fig:gw:iterative} (bottom right), the four principal angles between two iterative subspaces averaged on $10$ realizations are plotted. Around $40$ iterations seem sufficient to converge which corresponds to around $17,500$ likelihood evaluations with only~$6,000$ of them being used to compute the covariance matrix approximation. The three other plots represent the convergence of the quantities involved in the bound derived in Prop.~\ref{prop:final-bound}, that is to say $\mathbb{V}\mathrm{ar}_{P_{\Pi_r}}\left(\sqrt{\omega}\right)$ (or equivalently $\E_{P_{\Pi_r}}\left(\sqrt{\omega}\right)$) and $\E_{\post[\Pi_r]} \left( \mathbb{V}\mathrm{ar}_{\prior[\perp|r]}(\omega(X)) \right)$. The results stabilize within an acceptable band starting at step 40, and have converged after 60 steps, as for the principal angles.
This trend is less clear in individual runs, making it difficult to define a stopping criterion directly. Nevertheless, a threshold-based criterion could be adapted, \textit{e.g.} stopping when variance terms fall below $0.25$ and the expectation term exceeds $0.85$. 
The log-permeability fields obtained using the first four eigenvectors at iteration $40$ are similar to the \gls{cis} approach with prior samples plotted in Fig.~\ref{fig:gw:eigenspace}, and are therefore not shown for brevity.
\begin{figure}
\centering
\includegraphics[width=0.32\textwidth]{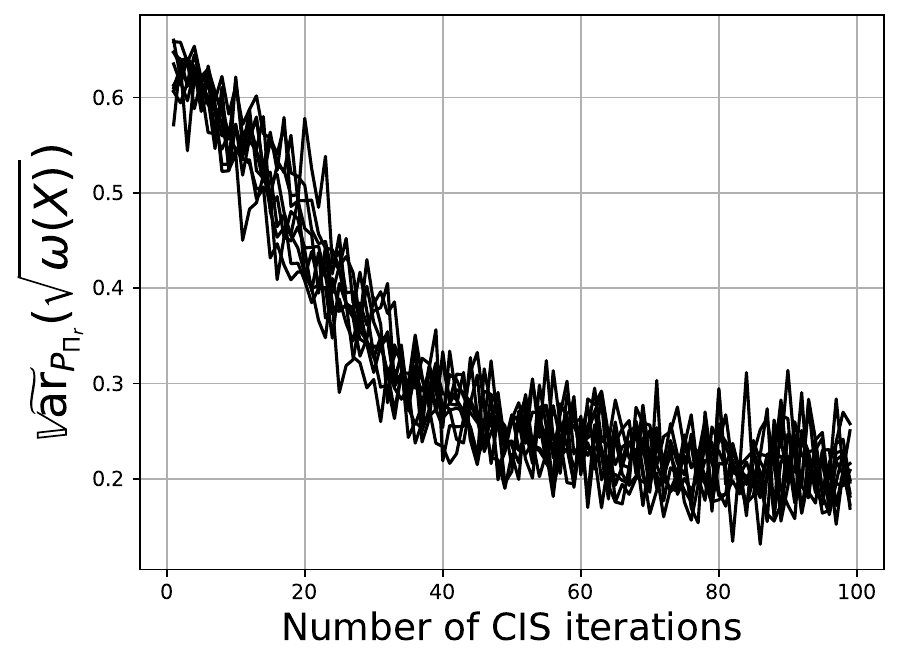}
\hspace{1cm}
\includegraphics[width=0.32\textwidth]{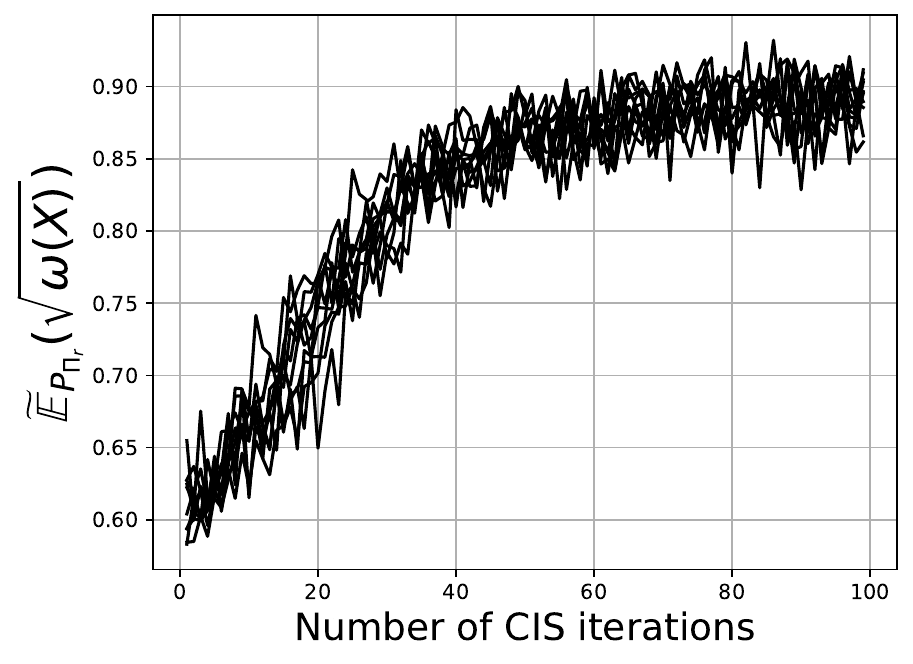}
\\
\includegraphics[width=0.32\textwidth]{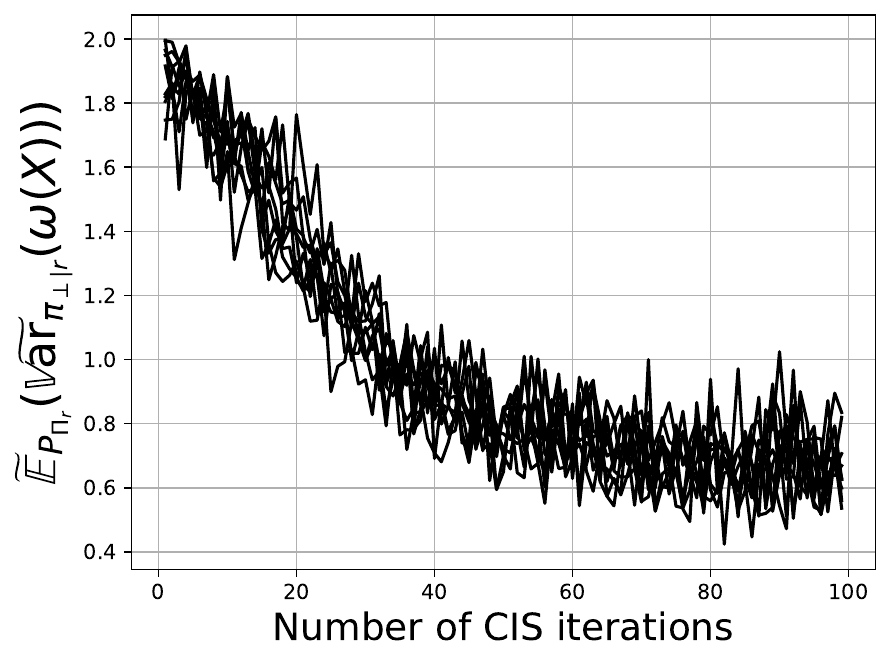}
\hspace{1cm}
\includegraphics[width=0.32\textwidth]{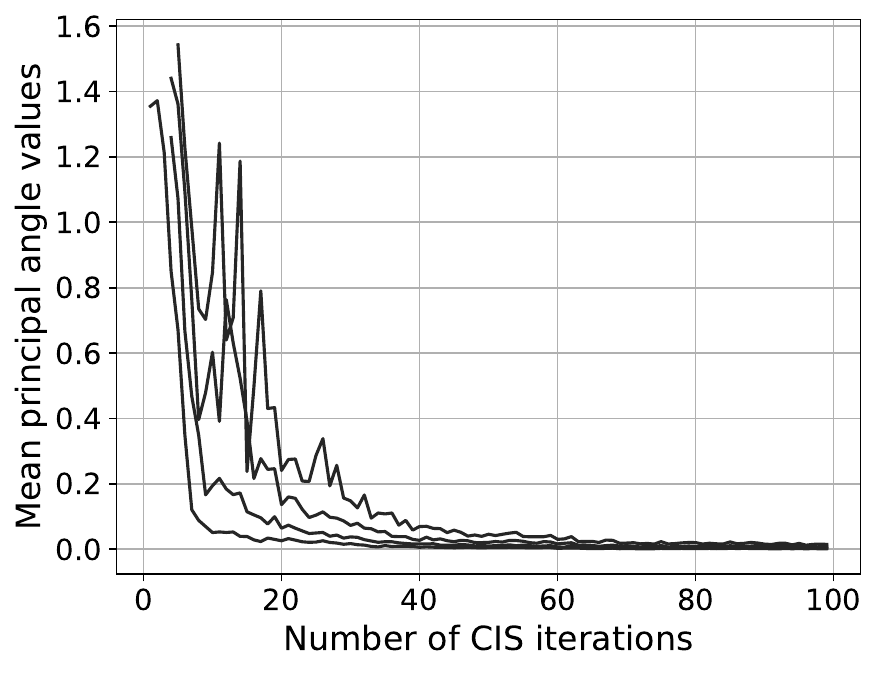}
\caption{Groundwater case - Convergence of the iterative \gls{cis} algorithm. 
\gls{mc} approximations of (top left) the expectation, (top right) the variance of the square root of the weights (bottom left) the expected conditional variance of the weights given the reduced variables over the approximate posterior distribution according to the \gls{cis} iteration. Results are plotted for $10$ independent runs of the iterative \gls{cis} algorithm. \gls{mc} computations are estimated using $80$ samples in the informed subspace, each one completed with $3$ samples in the non-informed subspace so that the variance and expectations are estimated along the current approximate posterior. (bottom right) Four principal angles considering two successive subspaces along algorithm iterations (averaged on $10$ runs).}\label{fig:gw:iterative}
\end{figure}
 
\subsection{Dimension reduction efficiency}\label{subsection:gw:efficiency}
In this section, we analyze the quality of the \gls{dr} given by the three methods \begin{inparaenum}[i)] \item using $\widetilde{\bs{H}}_P$ computed with $5,000$ prior samples (denoted prior-H), \item using $\widetilde{\bs{C}}_P$ computed with $25,000$ prior samples (denoted prior-C) and \item using $\widetilde{\bs{C}}_P$ computed with $40$ iterations of the iterative algorithm (denoted ite-C)\end{inparaenum}.
For each method and for different reduced dimension $r$,~$100 \times 100$ approximate posterior samples~$\{\bx^{(i,j)}\}_{i,j=1}^{100} \sim P_{\Pi_r}$ are generated by independent draws in the informed and non-informed subspaces, 
\begin{equation}\label{eq:approx-fullpost}
 \bx^{(i,j)} = \bs{V}_r \bs{z}_r^{(i)} + \bs{V}_\perp \bs{z}_\perp^{(j)}, \text{ with } \bs{z}_r^{(i)} \sim P_{r} \text{ and } \bs{z}_\perp^{(j)} \sim \prior_{\perp}.
\end{equation}
This construction exploits the key property of the constructed subspace decomposition: dimensions orthogonal to the informed subspace remain at their prior distribution. The approximate distribution of the field is plotted in Figure~\ref{fig:gw:spatialdistrib} at three selected locations (red stars on Figure~\ref{fig:gw:priorreal}). We first observe that the three methods produce comparable distributions, so that only the results for the iterative \gls{cis} algorithm is shown here. As expected, the posterior and prior distributions at the point $\bs{s}_1$ near the left boundary are similar, as the sensors on the right boundary are not sufficient to provide information on the left side of the domain. In contrast, the point $\bs{s}_2$ near the right boundary exhibits a posterior that is shifted and contracted relatively to the prior. This information is well captured by the informed subspace. A notable case is the point $\bs{s}_3$, where the posterior is slightly shifted relatively to the prior. For this point, none of the four-dimensional informed subspace captures this evolution from the prior.
\begin{figure}
\centering
\includegraphics[width=0.3\textwidth]{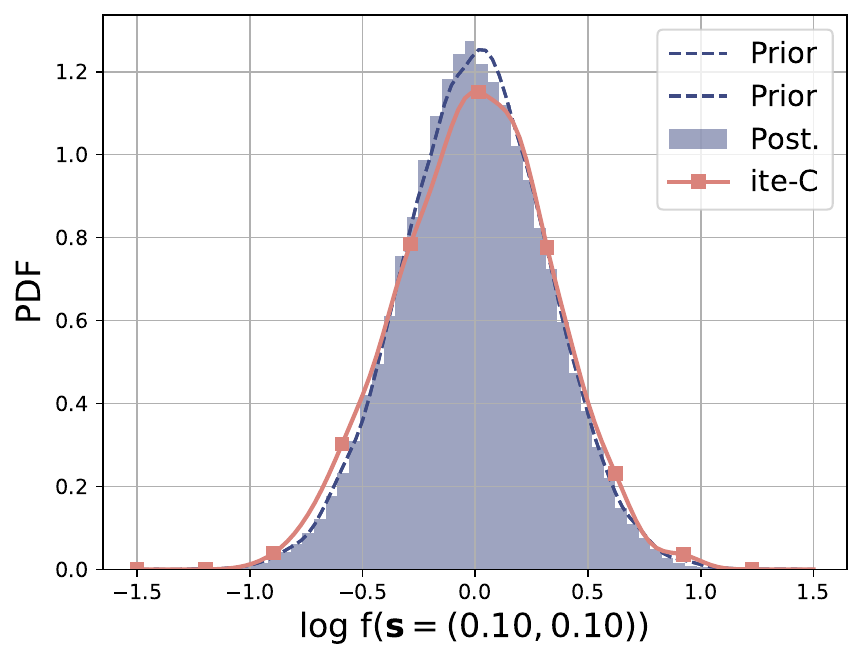}\hfill
\includegraphics[width=0.3\textwidth]{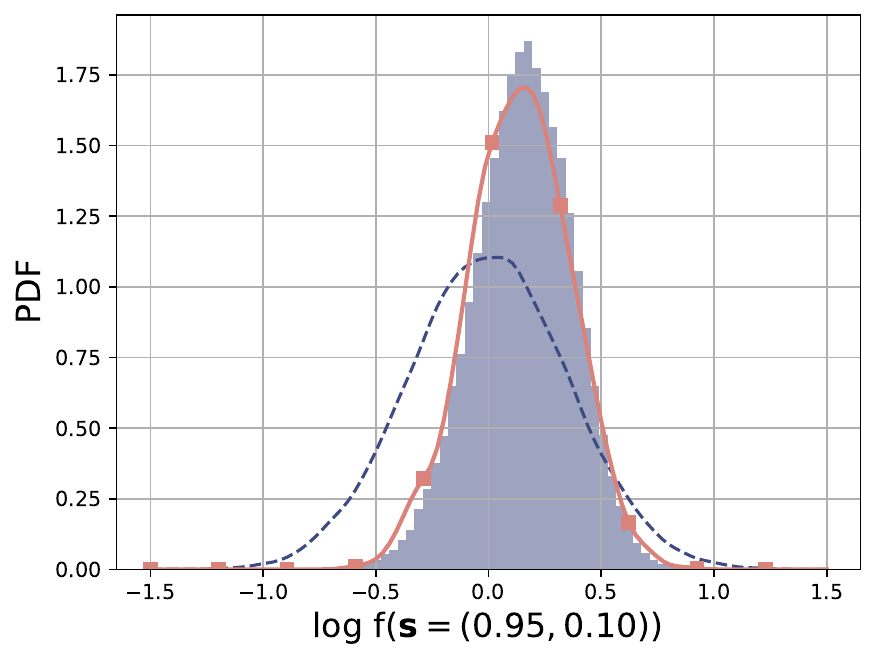}\hfill
\includegraphics[width=0.3\textwidth]{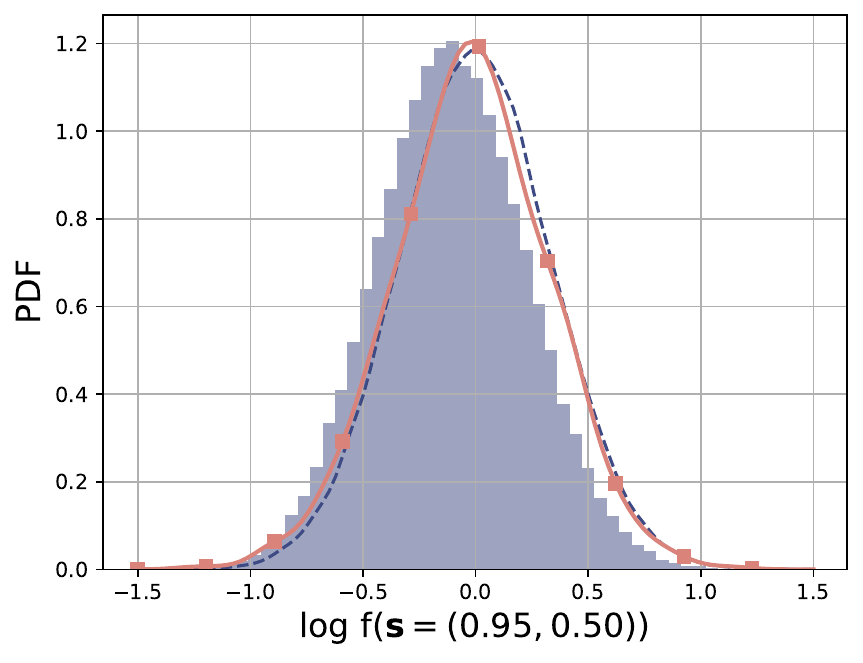}
\caption{Groundwater case - Distribution of the log-field value at three selected locations (see red stars on Figure~\ref{fig:gw:priorreal}). (left) $\bs{s}_1 = (0.1,0.1)$; (middle) $\bs{s}_2 = (0.95,0.1)$; (right) $\bs{s}_3 = (0.95, 0.5)$. Prior distribution (dashed blue); posterior distribution (filled blue); iterative \gls{cis} $(r=4)$ (pink squares). The posterior distributions for \gls{gis} and \gls{cis} with prior samples are similar to iterative \gls{cis} distribution.}\label{fig:gw:spatialdistrib}
\end{figure}

Figure~\ref{fig:gw:approxhellinger} (left and middle) shows the mean and the variance of the weights' square root as a function of the reduced dimension $r$. The expectation of the conditional variance behaves similarly as the variance and is not shown for brevity. As expected, the variances converge to zero when the entire space is considered as informed, while the expectation converges to one. While the values are similar for the first few dimensions, the gradient-based method exhibits a faster convergence than for the gradient-free method. It can be explained by the variance reduction identifying the most informative directions, while failing to accurately rank the less informative ones. The variance resulting from the iterative algorithm is very similar to the one obtained with prior sampling, demonstrating the capability of the iterative method to recover the subspace of interest.
\begin{figure}
\centering
\includegraphics[width=0.3\textwidth]{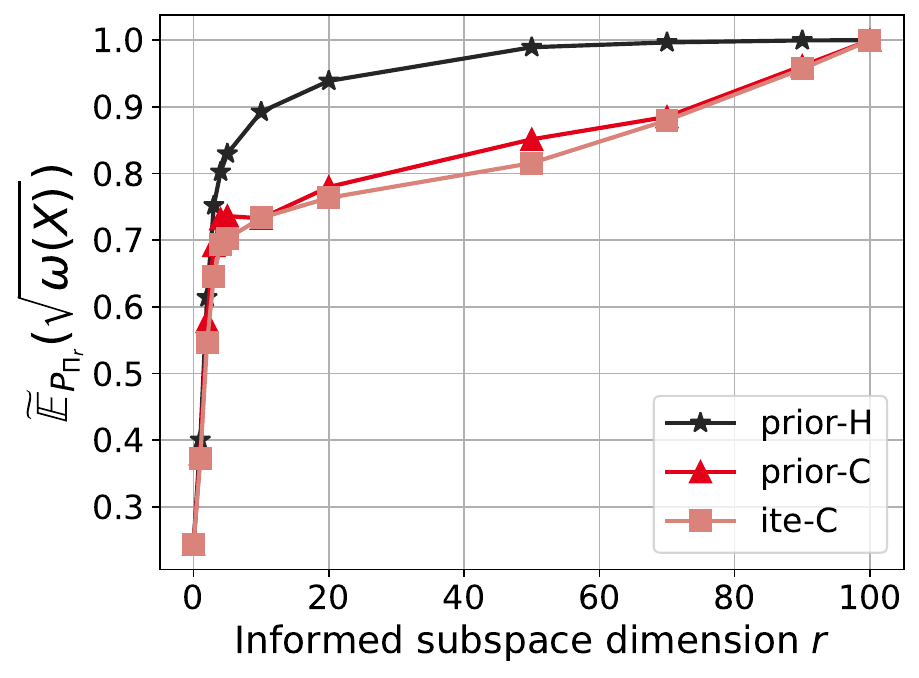}\hfill
\includegraphics[width=0.3\textwidth]{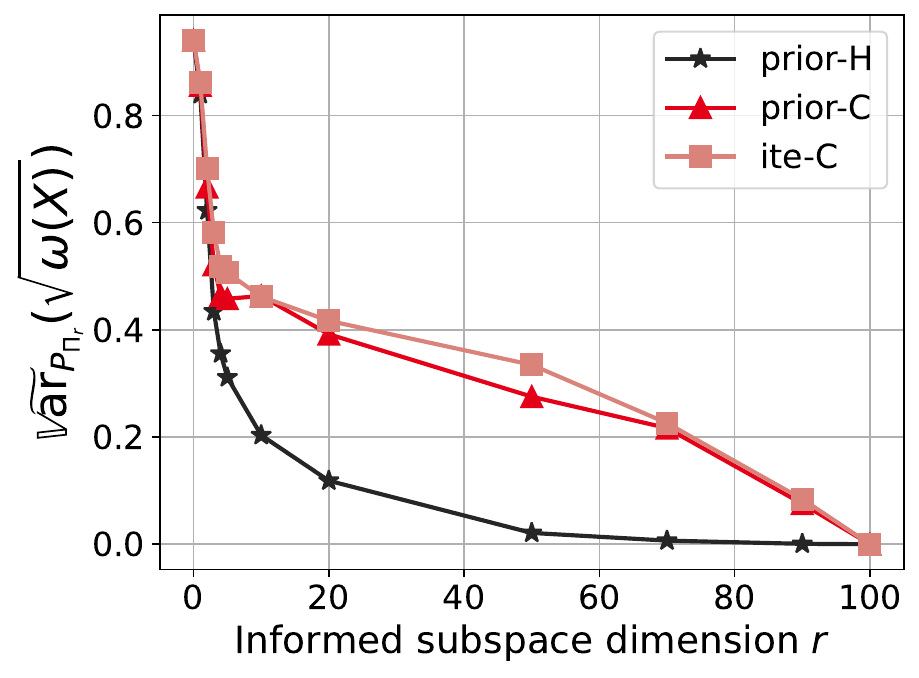}\hfill
\includegraphics[width=0.3\textwidth]{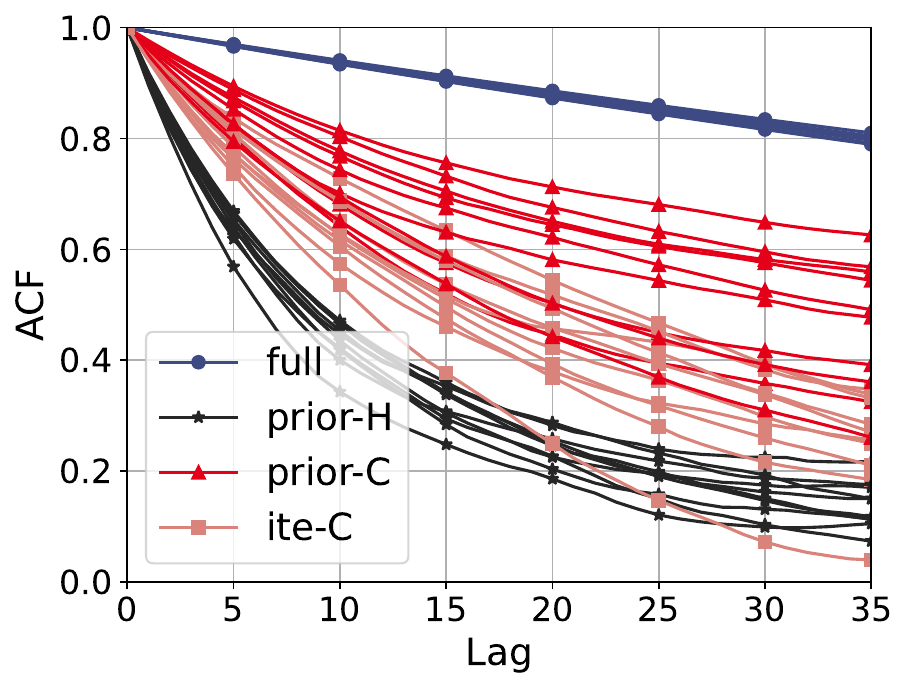}
\caption{Groundwater case - \gls{mc} approximations according to the reduced dimension $r$ of (left) the expectation, (middle) the variance considering the square root of the weights over the approximate posterior distribution. Computation using~$100$ draws in the informed subspace, each one completed with~$100$ draws in the non-informed subspace, which corresponds to a total of~$10,000$ draws.
(right) Autocorrelation function for some random dimensions variables. $r=4$. Full posterior sampling (dark blue circles); using the informed subspace obtained with the \gls{gis} method and prior samples (black stars);  using the informed subspace obtained with the \gls{cis} method and prior samples (red triangles);  using the informed subspace obtained with the \gls{cis} method iterative algorithm (pink squares).}\label{fig:gw:approxhellinger}
\end{figure}

We also use the informed subspaces of dimension $4$ to apply the delayed algorithm~\ref{algo:delayed} and obtain samples from the true posterior distribution. In Fig.~\ref{fig:gw:approxhellinger} (right), the autocorrelation functions of the samples along iterations are plotted for selected directions, comparing a naive full posterior sampling (dark blue circles) with the reduced posterior samplings. It is clear that considering a reduced-dimensional informed subspace improves the convergence and the mixing properties of the samples chain. Moreover, the informed subspace found by the iterative algorithm seems to offer a more efficient sampling.

%% file: section5.tex
\section{Application to atmospheric observations}\label{section:gomos}
We now consider a realistic test case from atmospheric remote sensing (\gls{gomos}) where the posterior distribution is known to have a substantially reduced uncertainty compared to the prior. Unlike the groundwater case in Section~\ref{section:gw}, this strong data informativeness leads to severe weight degeneracy when using naive weighted covariance estimation from the prior, making standard \gls{cis} initialization unsuitable. This test case therefore assesses whether the iterative \gls{cissmc} construction in Algorithm~\ref{algo:smc-cis} can successfully build the informed subspace without relying on a prior-based covariance estimation. The observation model is introduced in Section~\ref{subsection:gomos:model} and results are presented in Section~\ref{subsection:gomos:results}.

\subsection{GOMOS observation model}\label{subsection:gomos:model}
The \acrfull{gomos} instrument observes stellar light passing through Earth's atmosphere at different wavelengths ($\lambda$) and tangent altitudes ($\mathrm{alt}$). As light traverses the atmosphere, atmospheric gases absorb specific wavelengths, reducing the observed intensity. The measured quantity in this model is the light transmission $T_{\lambda,\mathrm{alt}}$ that is the ratio of observed to unobstructed intensity and that depends on the vertical concentration profiles $\{f^i\}_{1\leq i \leq N_g}$ of $N_g = 4$ atmospheric gases. The inverse problem consists in estimating these $N_g$ concentration profiles from the transmission observations. The relationship between transmissions and gas profiles follows Beer's law, 
\begin{equation}\label{eq:gomos:forward}
T_{\lambda,\mathrm{alt}}(f^1, f^2, f^3, f^4) = \exp\left( -\Int_{\text{path(alt)}} \suml_{i=1}^{4} a_\lambda^{i}(z(\zeta))f^i(z(\zeta))d\zeta\right),
\end{equation}
where path(alt) denotes the line of sight through the atmosphere at tangent altitude alt, and $z(\zeta)$ gives the altitude as a function of position $\zeta$ along this path. The coefficient $a_\lambda^{i}$ represents the absorption cross-section of gas $i$ at wavelength $\lambda$ as a function of the altitude, which is a known spectroscopic quantity describing how strongly the gas absorbs light. Following the inverse problem used in~\cite{cui2014,zahm2022}, the transmissions and gas profiles are discretized using~$N_\mathrm{alt}=50$ vertical layers. The inverse problem seeks to estimate $N_\mathrm{alt} \times N_g = 200$ log-concentrations. The discretized forward model takes the compact form
\begin{equation}\label{eq:gomos:forward}
\bs{T} = \exp \left( -(\bs{A} \otimes \bs{L})\mathrm{vec}(\bs{B}^\top)\right),
\end{equation}
where $\bs{T} \in \R^{N_\lambda\times N_\mathrm{alt}}$ contains the predicted transmission values,~$\bs{A}\in \R^{N_\mathrm{alt}\times N_\mathrm{alt}}$ is the geometry matrix with entries giving the path length through each atmospheric layer for each line of sight, ~$\bs{L}\in \R^{N_\lambda\times N_g}$ contains the absorption cross-sections, ~$\bs{B}\in \R^{N_\mathrm{alt}\times N_g}$ contains the unknown concentration values, $\otimes$ is the Kronecker product and $\mathrm{vec}$ is the vectorization operator $\mathrm{vec}$.
 
The prior distribution is also built similarly as in~\cite{cui2014,zahm2022}. The prior distributions of the four gases log-profiles are chosen to be Gaussian, with each profile modeled using a squared-exponential covariance kernel with correlation length 10 km. To isolate the \gls{dr} challenge from prior misspecification issues, the prior mean and variance amplitude are set using the true log-profiles statistics. The observation vector consists of $1,416$ wavelengths measured at each of the $50$ altitude layers, yielding $70,800$ total observations. Synthetic data are generated by evaluating the forward model~\eqref{eq:gomos:forward} at the true concentration profiles and adding independent Gaussian noise with known variance. Figure~\ref{fig:gomos:data} illustrates sample transmission observations at selected altitudes as a function of wavelength, showing both the noise-free true values and the noisy observed data. Figure~\ref{fig:gomos:distribution} (top row) displays the prior distribution for each gas log-profile across altitude, while the middle row shows the posterior obtained using a reference \gls{smc}-\gls{mcmc} sampler (see caption for details). As noted in~\cite{cui2014} and visible in Figure~\ref{fig:gomos:distribution}, the four gases exhibit varying levels of data informativeness: the first gas profile is strongly constrained by the observations, while the fourth remains largely uninformed.
\begin{figure} 
     \includegraphics[width=0.3\textwidth]{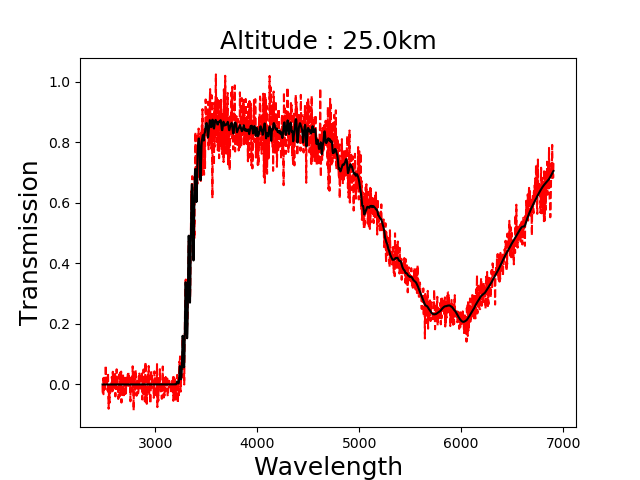}
     \includegraphics[width=0.3\textwidth]{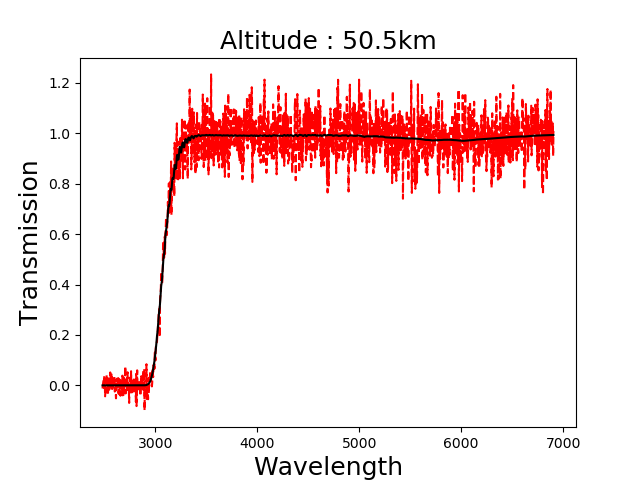}
     \includegraphics[width=0.3\textwidth]{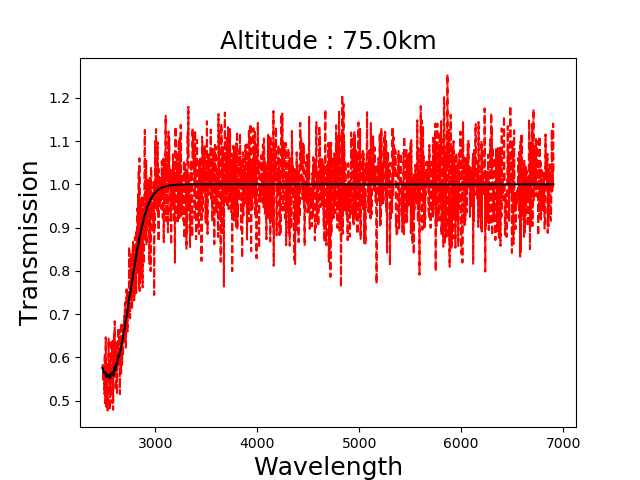}
\caption{GOMOS case - Observations: light transmissions at three altitudes (25km, 50km and 75km) depending on the wavelength. True values (black plain) and observed values (red dashed).}\label{fig:gomos:data}
\end{figure}

\subsection{CIS results}\label{subsection:gomos:results}
As a preliminary remark, Algorithm~\ref{algo:iterative-cis} is ineffective in this case due to weight degeneracy caused by the large information gain from the observations.
This justifies the iterative construction approach of Algorithm~\ref{algo:smc-cis} which uses tempered distributions to build progressively the projector and overcome the degeneracy. In the presented results, $N_\mathrm{samp}=50$ particles are used, with $10$ \gls{pcn} moves at each tempering iteration. For each particle in the reduced space, we then draw $100$ samples in the non-informed space to approximate the current weighted covariance on the full space. The reduced space dimension $r$ is chosen such that $\lambda_r \leq 0.6$, and a maximum dimension of $40$ is imposed.
\begin{figure}
\centering 
\includegraphics[height=4.8cm]{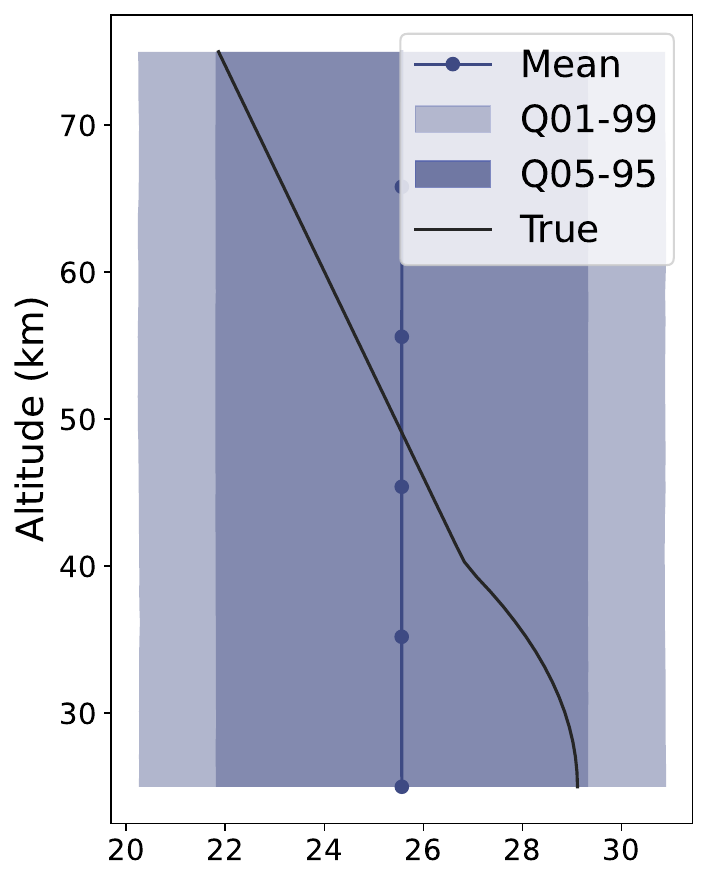}
\includegraphics[height=4.8cm]{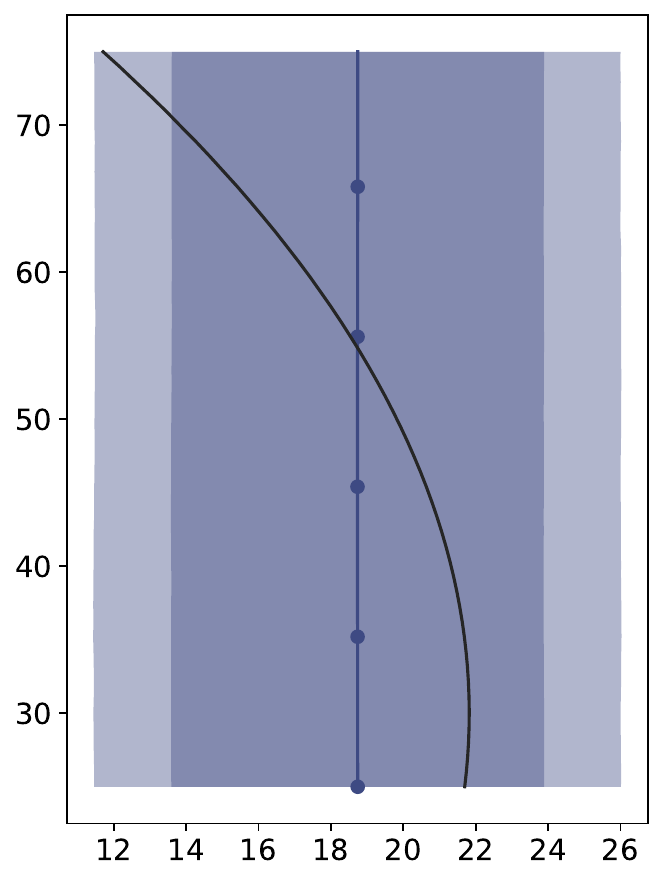}
\includegraphics[height=4.8cm]{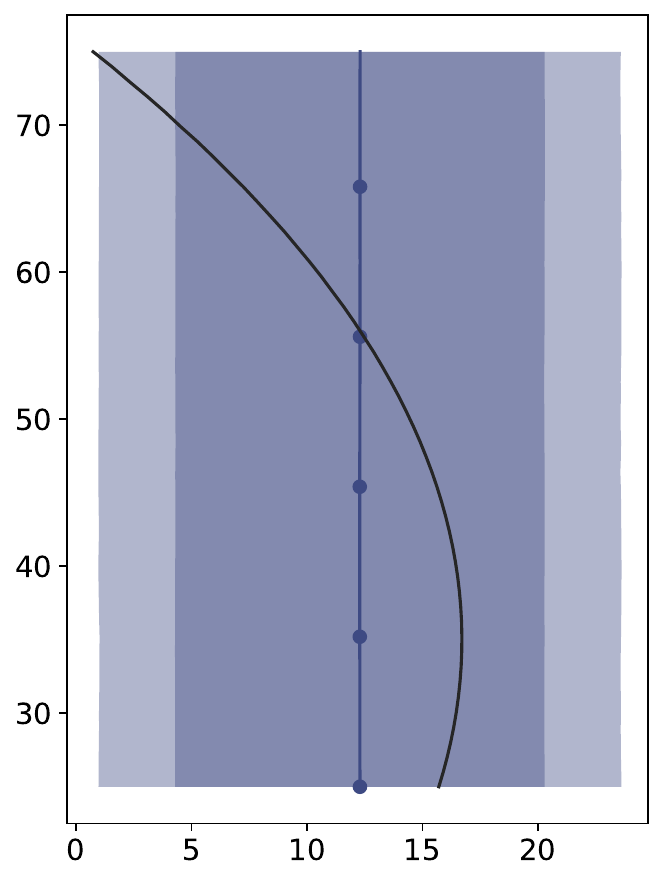}
\includegraphics[height=4.8cm]{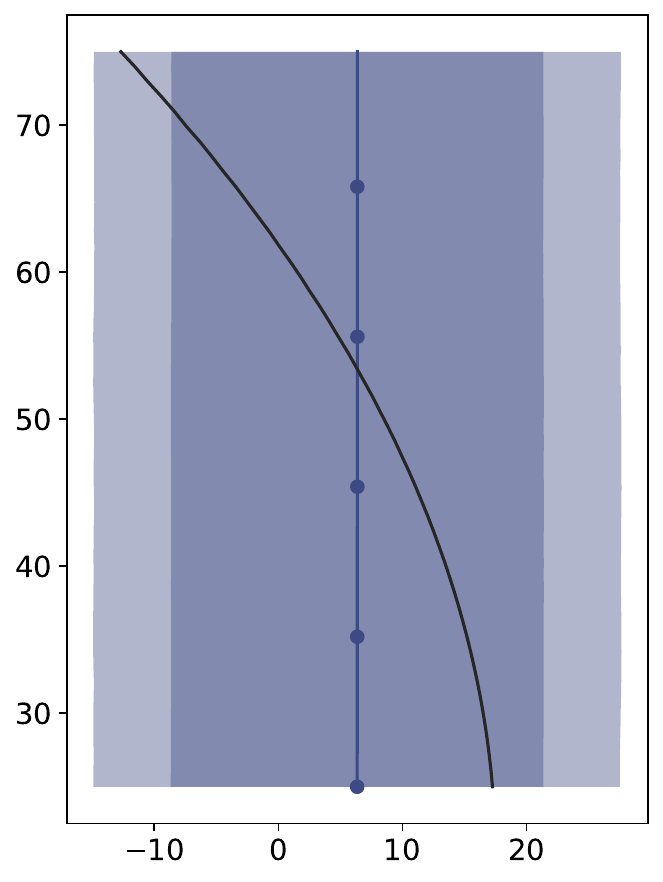}\\
\includegraphics[height=4.8cm]{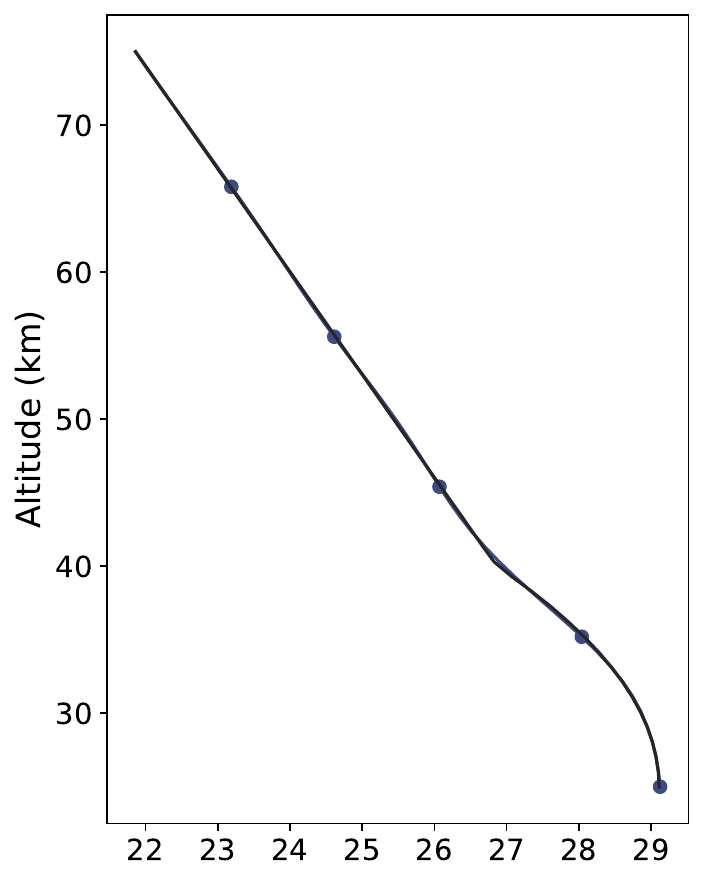}
\includegraphics[height=4.8cm]{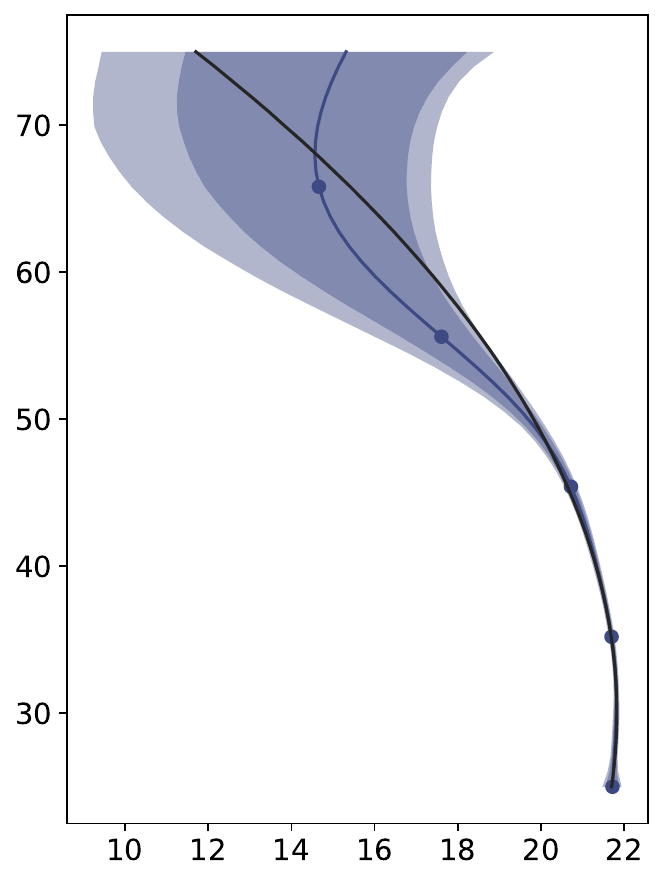}
\includegraphics[height=4.8cm]{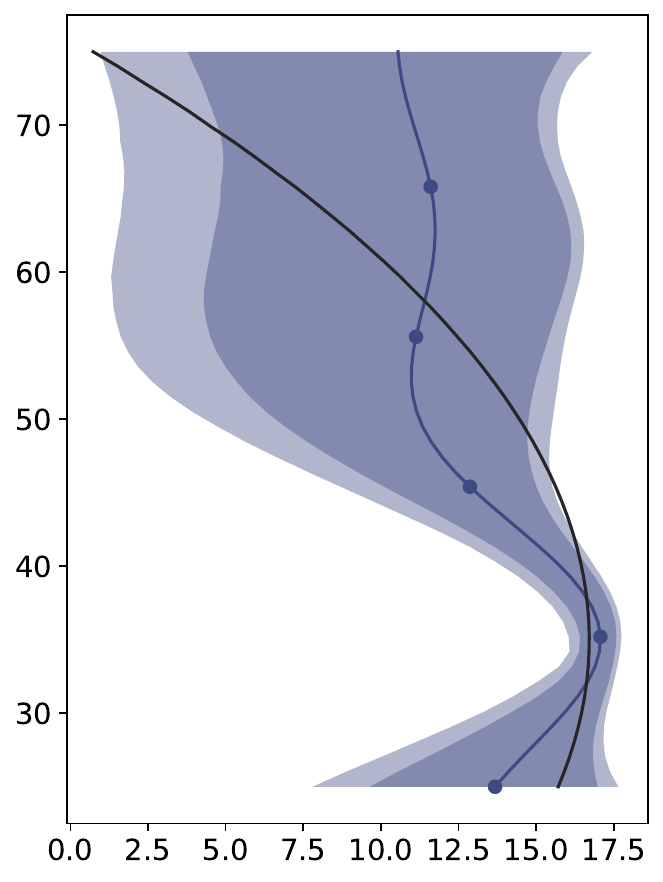}
\includegraphics[height=4.8cm]{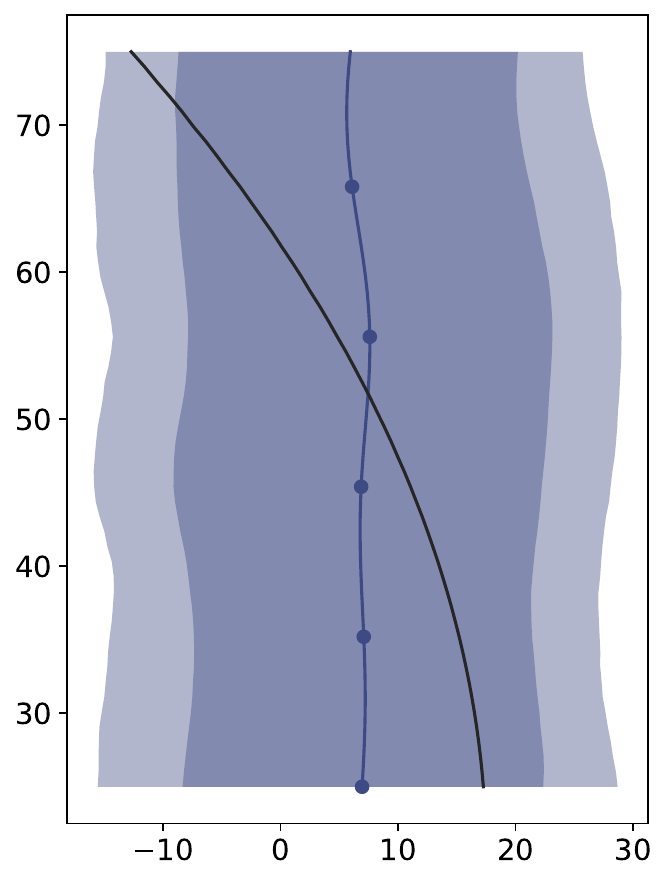}\\
\includegraphics[height=5cm]{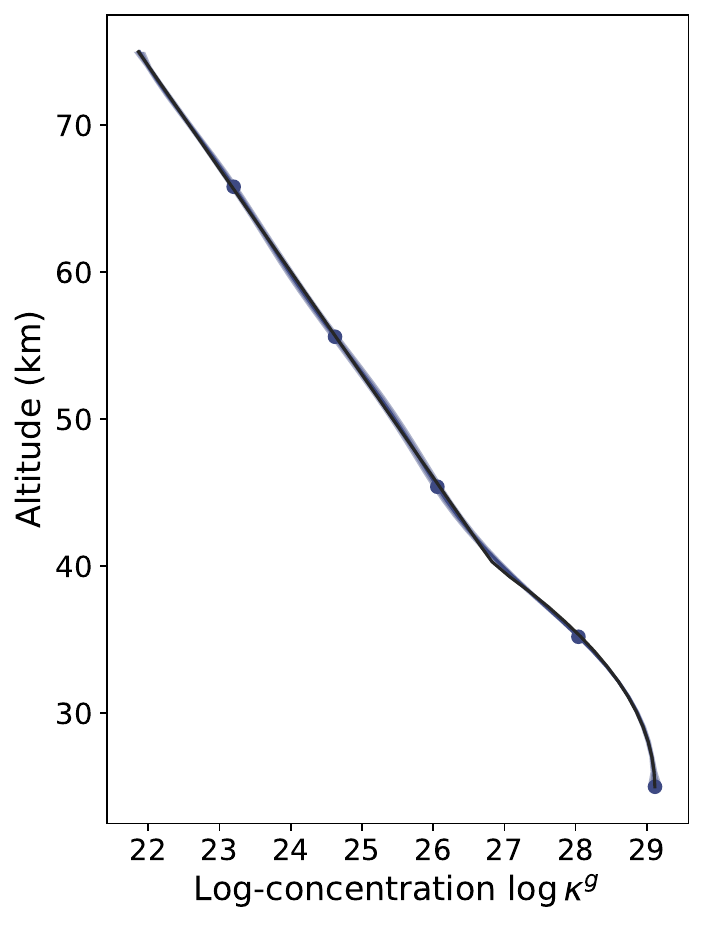}
\includegraphics[height=5cm]{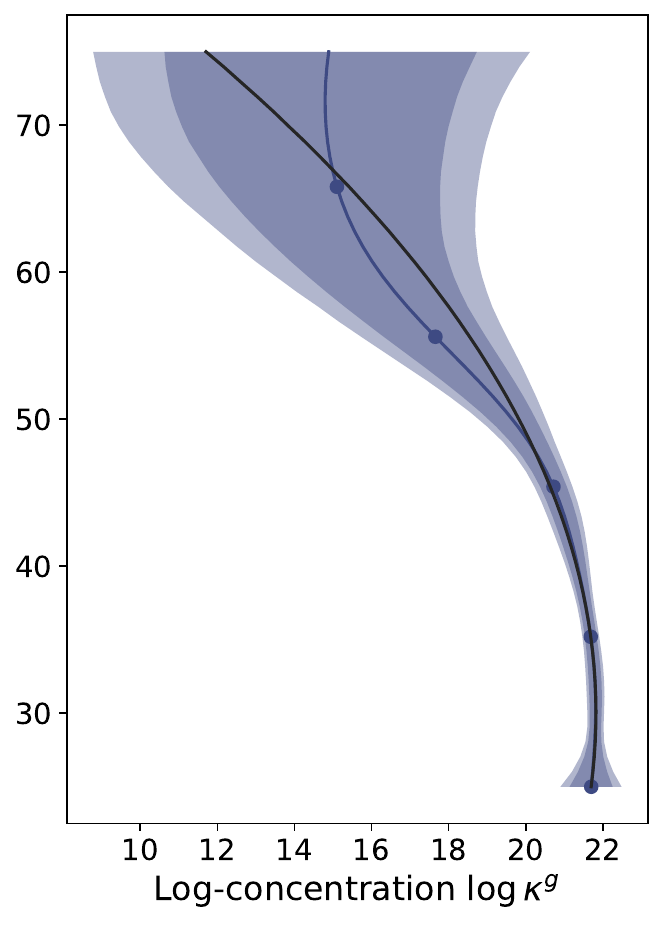}
\includegraphics[height=5cm]{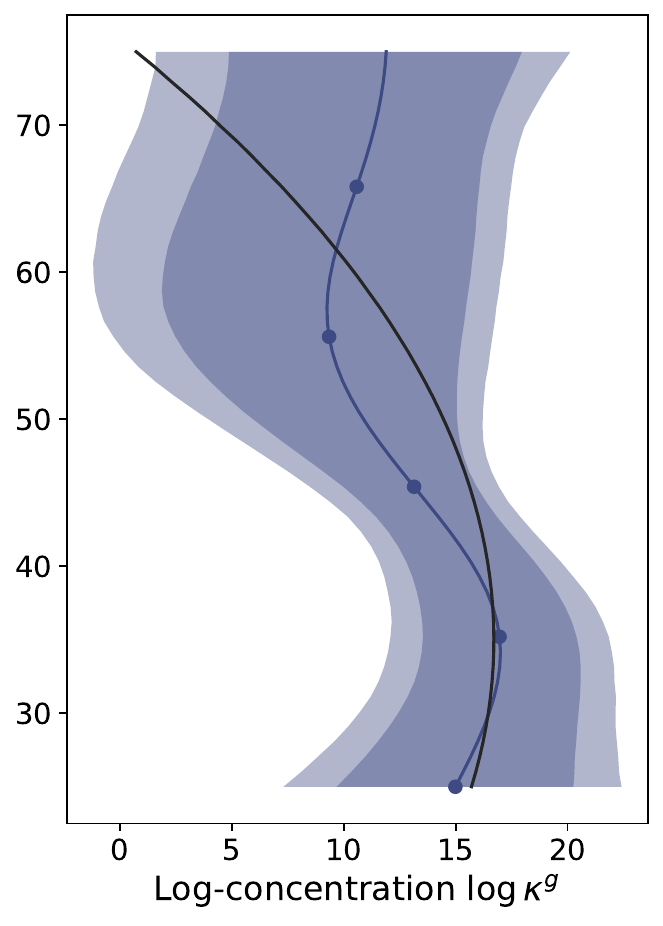}
\includegraphics[height=5cm]{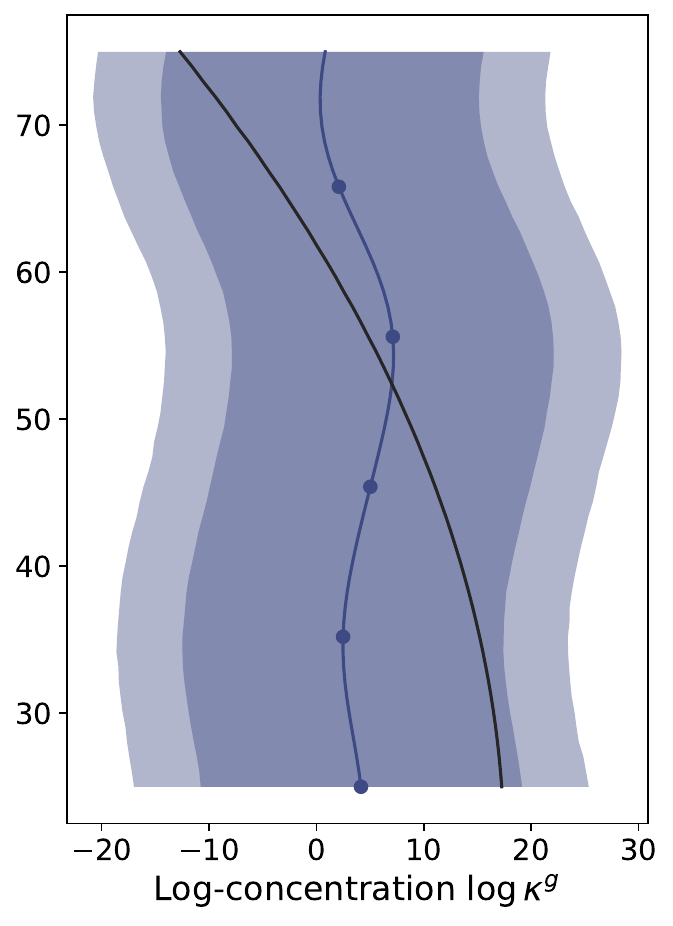}
\caption{GOMOS case - Log-profiles of the four gas and their distributions. (top row) Prior distribution; (middle row) Posterior distribution; 
Reference posterior obtained from a \gls{smc} algorithm with $2,500$ particles using adaptive tempering and a preconditioned Crank-Nicolson proposal as done in~\cite{carrera2024}, followed by three \gls{mcmc} chains initialized from the SMC output, yielding $3\times 2\times 10^5$ samples total.
(bottom row) approximate posterior distribution using \gls{cissmc} agorithm. True log-profile (black line); mean (blue circles); 05-99\% quantiles (darker area); 01-99\% quantiles (lighter area).}\label{fig:gomos:distribution}
\end{figure}
The eigenvalue spectrum (Figure~\ref{fig:gomos:smcres}, right) indicates that the informed subspace has dimension $r = 30$, which is only $15\%$ of the full $200$-dimensional parameter space. Analysis of the cumulative modal contributions to each gas profile (Fig.~\ref{fig:gomos:smcres}, left) confirms the variable levels of information observed in Figure~\ref{fig:gomos:distribution}: for the first gas, the leading modes capture nearly all variations, indicating strong observational constraints. In contrast, the fourth gas exhibits uniformly distributed contributions across modes, confirming it remains essentially uninformed by the data. The intermediate gases (second and third) show progressively decreasing informativeness.
\begin{figure}
\centering
\includegraphics[height=4.3cm]{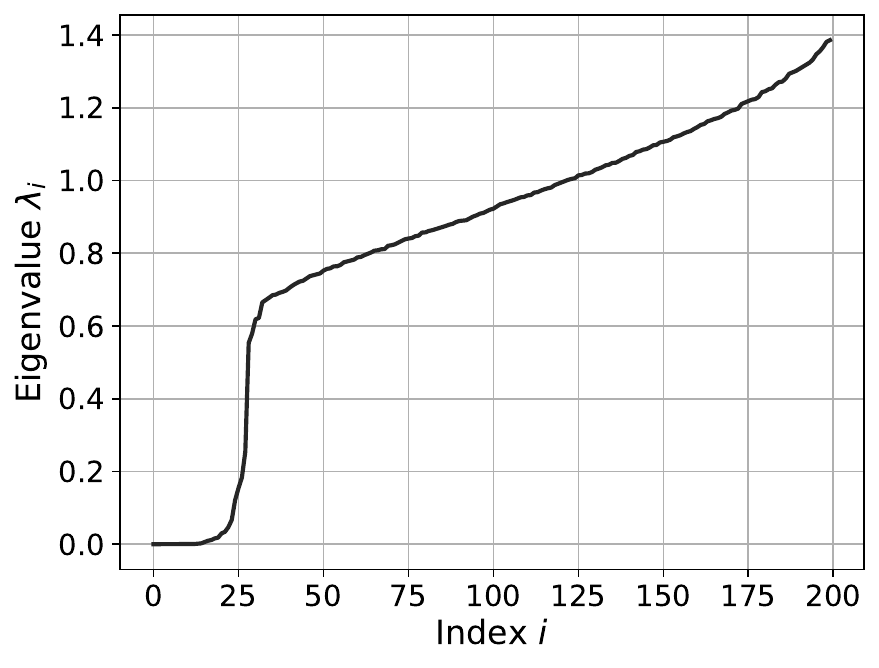}
\includegraphics[height=4.3cm]{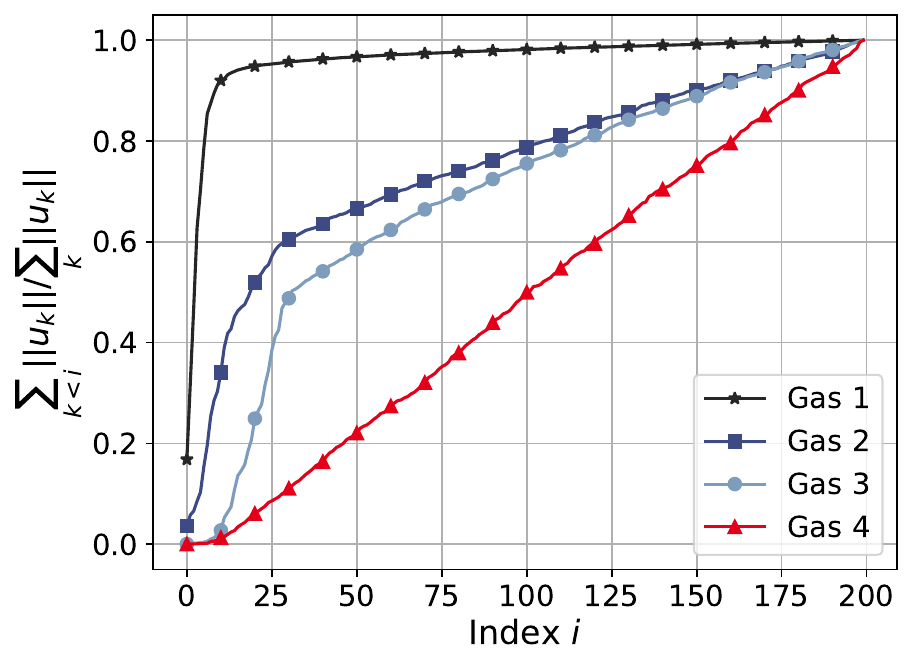}
\caption{(left) Eigenvalue of the \gls{cis} eigenproblem using the final covariance approximation obtained with the \gls{cissmc} algorithm. We recall that a small eigenvalue indicates an informed direction. (right) Cumulative contribution of the \gls{cis} modes to each gas.}\label{fig:gomos:smcres}
\end{figure}

Once the  informed subspace has converged, we sample the reduced space using the samples obtained from the last iteration of the iterative \gls{cissmc} algorithm and their covariance for initialization. The resulting reduced-space samples are combined with independent prior samples from the orthogonal complement to construct full-dimensional posterior approximations as detailed in Eq.~\eqref{eq:approx-fullpost}. Figure~\ref{fig:gomos:distribution} (bottom row) shows the resulting approximation, which closely matches the reference posterior (middle row) despite being constructed from only $30$ informed dimensions. The largest discrepancy between the two distributions occurs in the bottom informed area of the third gas, where the approximate distribution is less concentrated than the true posterior.

\subsection{Complexity}
In the \gls{gomos} case, the complete \gls{cissmc} procedure requires approximately $10^6$ likelihood evaluations, including the final \gls{mcmc} sampling. For comparison, the \gls{gis} method uses $10^5$ gradient computations (see~\cite{zahm2022}), each costing ten times more than a single likelihood evaluation. Thus, both approaches have comparable total computational cost, equivalent to around $10^6$ likelihood evaluations, while \gls{cissmc} has the advantage of being gradient-free.
These two applications (the groundwater and \gls{gomos} cases) demonstrate both the versatility and robustness of the \gls{cis} construction across problems with different characteristics of posterior-to-prior variance reduction regimes.

Table~\ref{tab:summary} summarizes the comparison of the performance
 of the proposed method with gradient-based approaches. As mentioned above, the \gls{cis} method requires only forward model evaluations, whereas gradient-based approaches require gradient computations, typically obtained using an adjoint code or \gls{fd}. In terms of likelihood evaluations, the two test cases show that the \gls{cis} method requires roughly $5$ to $10$ times more evaluations. However, the cost per iteration is higher for gradient-based approaches, ranging from approximately $2$ to $n$ forward model evaluations, making the \gls{cis} method particularly competitive compared to finite-difference methods, where the per-sample gradient cost grows with the parameter space dimension. Both methods can be parallelized to draw in the non-informed subspace. The sampling of the informed subspace is parallelizable when using \gls{cissmc} (resp. \gls{smc}) for the gradient-free (resp. gradient-based) approach. For the \gls{gis} method, gradient approximations are parallelizable as well.
\begin{table} 
    \centering
\begin{tabular}{p{3cm}|p{5cm}|p{5cm}}
     & Gradient-free (\gls{cis} method) & Gradient-based (\gls{gis} method) \\ \hline 
     Model solver & Forward model & Forward model + Adjoint method or \gls{fd} approximations \\
     Number of \gls{mcmc} evaluations & $\sim$ 5-10 K &  1 K  \\
     Cost per iteration &~$1$ likelihood & $\sim$ 2 likelihoods if adjoint code or $\sim$ $n+1$ likelihoods if \gls{fd} \\
     Parallelisation & Possible for~$Z_\perp$ draws (always) and particles Markov moves (\gls{cissmc}) & Possible for particles Markov moves (if \gls{smc}) and for gradient (if \gls{fd})
 \end{tabular}
\caption{Comparison of the \gls{gis} and \gls{cis} methods on a computational viewpoint. These order of magnitude aims at giving a general idea of which method to use.} \label{tab:summary}
\end{table}
From a practical perspective, the \gls{cis} method is particularly suitable when gradients are unavailable and remains applicable even for concentrated posteriors. The choice of computing \gls{cis} from prior samples or via \gls{cissmc} should be guided by a simple diagnostic: draw a small number of prior samples, compute their likelihood ratios, and examine the resulting weight distribution. Near-uniform weights indicate that prior samples are well-suited to the target, and direct prior sampling is sufficient. Disparate weights suggest that iterative refinement is preferable, as it progressively improves the proposal distribution at a controlled computational cost. Severely degenerate weights, where only a few number of samples carry most of the weights, indicate that the target is distant from the prior and that \gls{cissmc} is necessary. This empirical approach ensures that computational cost scales with problem difficulty: simple cases are handled cheaply, and the full \gls{cissmc} complexity is incurred only when needed.

%% file: section6.tex
\section{Conclusion}
This work introduces a new gradient-free projector for input \gls{dr} in Bayesian inverse problems.
The projector is based on the eigenelements of the pencil composed of the posterior and the prior covariance matrices. 
The key idea is to assume that the posterior distribution is informed along directions where it is contracted relative to the prior, i.e., where the posterior variance is most reduced. The contributions of this paper are threefold. First, from a theoretical perspective, we show that the covariance-ratio-based projector combined with the optimal likelihood approximation provides a better representation than the state-of-the-art formulation in the Bayesian linear Gaussian case. We extend the projector to nonlinear problems and derive theoretical bounds. Second, from an algorithmic perspective, we propose two procedures to construct the reduced subspace: one based on prior \gls{mcmc} samples for well conditioned problems, and another using \gls{smc} samples for problems with higher uncertainty reduction. Third, from a practical perspective, the methodology is illustrated on two application examples. In the groundwater flow problem, the covariance informed subspace recovers the main directions identified by a gradient-based approach, offering a gradient-free alternative with comparable accuracy. In the atmospheric gas application, even with a highly peaked posterior that causes weight degeneracy in the covariance estimation, the directions of interest can still be recovered by modifying the weights and coupling the method with a \gls{smc} scheme.

This work initiate several future research directions. First of all, using the covariance-ratio based projector in complex cases, such as \gls{gomos}, the implementation of an iterative projector construction based on \gls{smc} requires the tuning of hyperparameters. To make the approach more broadly applicable, systematic strategies for selecting these hyperparameters are needed. Moreover, the current indicator assumes that an informed direction corresponds to a reduction of posterior variance relative to the prior variance. In practice, this assumption may need to be relaxed, as observations can induce posterior shifts or bimodal behaviour. Capturing such effects will require the development of alternative indicators that go beyond second-moment estimators. Finally, the method could be extended to handle more general prior distributions and could serve as a basis for coupled input-output \gls{dr}.

%% file: appendixA.tex
\section{Proofs of the Hellinger distance bound}
In the following, the dependency to $Y$ is omitted in $\mc{L}$ and $\post$ notations for sake of brevity.

\subsection{Hellinger distance using weights}\label{appendix:hellinger-weights}
This Section presents the proof of Proposition~\ref{prop:hellinger-weights}.

\begin{proof}
We consider the squared Hellinger distance between the posterior and its approximation,
\begin{align}
D_H\left(P,P_{\Pi_r}\right)^2 &= \inv{2}\Int_{\R^n} \left(\sqrt{P(\bx)} - \sqrt{P_{\Pi_r}(\bx)}\right)^2 d\bx,\\
&= 1 - \Int_{\R^n} \sqrt{P(\bx)P_{\Pi_r}(\bx)} d\bx,\notag \\
&= 1 - \Int_{\R^n} \sqrt{P(\bx)/P_{\Pi_r}(\bx)} dP_{\Pi_r}(\bx),\notag \\
&= 1 - \Int_{\R^n} \sqrt{\omega(\bx)} dP_{\Pi_r}(\bx) = 1 - \E_{P_{\Pi_r}}\left(\sqrt{\omega(\bx)}\right).
\end{align}
Using the fact that the expectation of the weights equals one, we also have
\begin{align}
\mathbb{V}\mathrm{ar}_{P_{\Pi_r}}(\sqrt{\omega(\bx)}) = \E_{P_{\Pi_r}}\left(\omega(\bx)\right) - \E_{P_{\Pi_r}}\left(\sqrt{\omega(\bx)}\right)^2 = 1 - \E_{P_{\Pi_r}}\left(\sqrt{\omega(\bx)}\right)^2,
\end{align}
so that
\begin{equation}
D_H\left(P,P_{\Pi_r}\right)^2 = 1 - \sqrt{1 - \mathbb{V}\mathrm{ar}_{P_{\Pi_r}}(\sqrt{\omega(\bx)})}.
\end{equation}
\end{proof}
Minimizing the Hellinger distance between the posterior and its approximation is therefore equivalent to minimizing the variance, or equivalently, maximizing the expectation of the square root of the weights under the approximate posterior distribution.

\subsection{Bound of the final approximation quality}\label{appendix:hellinger-bound}
This bound result is derived from~\cite{cui2022a}~(Theorem~3.1 and its proof) that is recalled in the following proposition.
\begin{proposition}[Monte Carlo likelihood approximation]\label{prop:mc-likelihood-approx}
The expected Hellinger distance between $P_{\Pi_r}$ and its \gls{mc} approximation $P_{\Pi_r}^{(N)} \propto \mc{L}_{\Pi_r}^{(N)}\pi$ over the Monte Carlo draws $\{\bs{z}^{(i)}\}_{1\leq i \leq N}$ is bounded with
\begin{equation}
    \E_{MC}\left( d_H(P_{\Pi_r}, P_{\Pi_r}^{(N)})^2 \right) \leq \frac{2}{\pi_\mathrm{data}(Y) N}  \Int_{\R^r} \frac{\mathbb{V}\mathrm{ar}_{\pi_{\perp|r}}(\mc{L}(\bs{V}_r \bs{z}_r + \bs{V}_\perp Z_\perp))}{\mc{L}_{\Pi_r}(\bs{z}_r)} \pi_{r}(\bs{z}_r) d\bs{z}_r,
\end{equation}
\end{proposition}
Using the fact that 
\begin{align*}
\mathbb{V}\mathrm{ar}_{\prior[\perp|r]}\left( \omega(X) \right) &= \E_{\prior[\perp|r]}\left[ \omega(X)^2 \right] - \E_{\prior[\perp|r]}\left[ \omega(X) \right]^2 \\
&= \E_{\prior[\perp|r]}\left[ \frac{\mc{L}(X)^2}{\mc{L}_{\Pi_r}(Z_r)^2} \right] - \E_{\prior[\perp|r]}\left[ \frac{\mc{L}(X)}{\mc{L}_{\Pi_r}(Z_r)} \right]^2 \\
&= \inv{\mc{L}_{\Pi_r}(Z_r)^2} \mathbb{V}\mathrm{ar}_{\prior[\perp|r]} \left( \mc{L}(X) \right),
\end{align*}
the bound of Proposition~\ref{prop:mc-likelihood-approx} is rewritten depending on the weights variance:
\begin{align*}
    \E_{MC}\left( d_H(\post[\Pi_r], \post[\Pi_r]^{(N)})^2 \right) &\leq \frac{2}{\ev N}  \Int_{\R^r} \mathbb{V}\mathrm{ar}_{\prior[\perp|r]}(\omega(\bs{V}_r \bs{z}_r + \bs{V}_\perp Z_\perp)) \mc{L}_{\Pi_r}(\bs{z}_r) \prior[r](\bs{z}_r) d\bs{z}_r \\
    &= \frac{2}{N} \E_{\post[\Pi_r]} \left( \mathbb{V}\mathrm{ar}_{\prior[\perp|r]}(\omega(X)) \right).
\end{align*}
The more accurate the approximation, the lesser the variance of the weights, and therefore the lesser the number of samples needed to have a good approximation.
Using the triangle inequality for the Hellinger distance, we get 
\begin{align*}
d_H(\post, \post[\Pi_r]^{(N)})^2 &\leq d_H(\post, \post[\Pi_r])^2 + d_H(\post[\Pi_r], \post[\Pi_r]^{(N)})^2, \\
\E_{MC}\left(d_H(\post, \post[\Pi_r]^{(N)})^2\right) &\leq \E_{MC}\left(d_H(\post, \post[\Pi_r])^2\right) + \E_{MC}\left(d_H(\post[\Pi_r], \post[\Pi_r]^{(N)})^2\right), \\
\E_{MC}\left(d_H(\post, \post[\Pi_r]^{(N)})^2\right) &\leq d_H(\post, \post[\Pi_r])^2 + \E_{MC}\left(d_H(\post[\Pi_r], \post[\Pi_r]^{(N)})^2\right),
\end{align*}
so that 
\begin{equation}
\E_{MC}\left(d_H(\post, \post[\Pi_r]^{(N)})^2\right) \leq 1 - \sqrt{1 - \mathbb{V}\mathrm{ar}_{P_{\Pi_r}}(\sqrt{\omega(\bx)})} + \frac{2}{N} \E_{\post[\Pi_r]} \left( \mathbb{V}\mathrm{ar}_{\prior[\perp|r]}(\omega(X)) \right).
\end{equation}
\hfill \cqfd

%% file: appendixB.tex
\section{Posterior variance reduction}\label{appendix:reduc-variance}
The strong assumption required to use the \gls{cis} approach is the fact that the posterior variance is small relatively to the prior variance for informed directions. In this appendix, in Section~\ref{subappendix:reduc-variance:conditions}, we show that the \gls{blg} case satisfies this assumption and we present some general conditions ensuring the reduction of the prior variance. In Section~\ref{subappendix:reduc-variance:examples}, some counter-examples, where the informed direction has a variance equal or greater than its prior, are presented. 

\subsection{Conditions to have a reduced posterior variance}\label{subappendix:reduc-variance:conditions}
We are looking for conditions on which~$\var_{\post}(X) \preceq \var_{\prior}(X)$. First of all, using the law of total variation,
\begin{align}
\var(X) &= \E(\var(X|Y)) + \var(\E(X|Y)), \\
\text{\textit{i.e.} } \bs{C}_{\prior} &= \E_Y(\bs{C}_{\post}(Y)) + \var_Y(\bs{\mu}_{\post}(Y)).
\end{align}
In the \gls{blg} case,~$\bs{C}_{\post}$ does not depend on~$Y$. Since the variance is always non-negative, we directly have~$\bs{C}_{\post} \preceq \bs{C}_{\prior}$. However, for more complex cases, the inequality is true only in average along the observations: we still could find observations for which it is not true. If we only are interested in a data-free approach, determining the directions along which the variance posterior is reduced relative to the prior one in average is sufficient.

In the following, the covariance matrices are expressed using the gradient of the likelihood, in order to derive conditions on the forward model for the assumption to hold, assuming a Gaussian prior.
We want to express 
\begin{equation}
\var_{\post}(X) = \E_{\post}( XX^\top ) - \E_{\post}(X)\E_{\post}(X)^\top,
\end{equation}
which is equivalent to~$\E_{\post}( (X-\bs{\mu}_{\prior})(X-\bs{\mu}_{\prior})^\top ) - \E_{\post}(X-\bs{\mu}_{\prior})\E_{\post}(X-\bs{\mu}_{\prior})^\top$, using the derivatives of the likelihood, knowing that~$\post(X) = \mc{L}(X;\by)\prior(X)/\ev[\by]$.

Let us first compute the derivatives of the prior and the likelihood.
\begin{remark}[Derivatives of the prior]\label{rk:prior-deriv}
Considering a Gaussian prior~$X \stackrel{\prior}{\sim} \mc{N}(\bs{\mu}_{\prior},\bs{C}_{\prior})$, the derivatives of its p.d.f. write 
\begin{align}
    &\nabla \prior(X) = -\bs{C}_{\prior}^{-1}(X-\bs{\mu}_{\prior})\prior(X), \label{eq:grad-prior}\\
    &\nabla^2 \prior(X) = -\bs{C}_{\prior}^{-1}\prior(X) + \bs{C}_{\prior}^{-1}(X-\bs{\mu}_{\prior})(X-\bs{\mu}_{\prior})^\top \bs{C}_{\prior}^{-1} \prior(X)\label{eq:hess-prior}.
\end{align}
\end{remark}
\begin{remark}[Derivatives of the likelihood]\label{rk:likelihood-deriv}
Considering a likelihood~$\mc{L}(X;\by)$ and introducing the misfit~$J(X;\by)~=~-\log \mc{L}(X;\by)$, the derivatives of the likelihood along~$X$ write
\begin{align}
    &\nabla \mc{L}(X;\by) = -\nabla J(X;\by) \mc{L}(X;\by), \label{eq:grad-likelihood}\\
    &\nabla^2 \mc{L}(X;\by) = -\nabla^2 J(X;\by) \mc{L}(X;\by) + \nabla J(X;\by) \nabla J(X;\by)^\top \mc{L}(X;\by) \label{eq:hess-likelihood}.
\end{align}
\end{remark}
First, the expectation term can be rewritten as
\begin{align}
\E_{\post}(X-\bs{\mu}_{\prior}) &= \inv{\ev[\by]}\Int_{\R^n} (\bx-\bs{\mu}_{\prior}) \mc{L}(\bx;\by) \prior(\bx) d\bx, \notag\\
&= -\inv{\ev[\by]}\bs{C}_{\prior} \Int_{\R^n} \mc{L}(\bx;\by) \nabla \prior(\bx) d\bx, &\text{Substituting with~\eqref{eq:grad-prior}} \notag\\
&= \inv{\ev[\by]}\bs{C}_{\prior} \Int_{\R^n} \nabla \mc{L}(\bx;\by) \prior(\bx) d\bx, &\text{Integration by parts} \notag\\
&= -\bs{C}_{\prior} \Int_{\R^n} \nabla J(\bx;\by) \post(\bx) d\bx, &\text{Substituting with~\eqref{eq:grad-likelihood}} \notag\\
&= -\bs{C}_{\prior} \E_{\post}(\nabla J(X;\by))
\end{align}
so that 
\begin{equation}
\E_{\post}(X-\bs{\mu}_{\prior})\E_{\post}(X-\bs{\mu}_{\prior})^\top = \bs{C}_{\prior} \E_{\post}(\nabla J(X;\by)) \E_{\post}(\nabla J(X;\by))^\top \bs{C}_{\prior}.
\end{equation}
Then, the quadratic term can also be rewritten as 
\begin{align}
&\E_{\post}( (X-\bs{\mu}_{\prior})(X-\bs{\mu}_{\prior})^\top )\notag\\
= \ & \inv{\ev[\by]} \Int_{\R^n} (\bx-\bs{\mu}_{\prior})(\bx-\bs{\mu}_{\prior})^\top \mc{L}(\bx;\by) \prior(\bx) d\bx, \notag\\
= \ & \inv{\ev[\by]} \bs{C}_{\prior} \Int_{\R^n} \mc{L}(\bx;\by) \left( \nabla^2 \prior(\bx) + \bs{C}_{\prior}^{-1}\prior(\bx)\right) d\bx \bs{C}_{\prior} \notag\\
= \ & \inv{\ev[\by]} \bs{C}_{\prior} \Int_{\R^n} \mc{L}(\bx;\by)  \nabla^2 \prior(\bx)  d\bx \bs{C}_{\prior} + \bs{C}_{\prior} \notag\\
= \ & \inv{\ev[\by]} \bs{C}_{\prior} \Int_{\R^n} \nabla^2 \mc{L}(\bx;\by) \prior(\bx)  d\bx \bs{C}_{\prior} + \bs{C}_{\prior} \notag\\
= \ & -\inv{\ev[\by]} \bs{C}_{\prior} \Int_{\R^n} \nabla^2 J(\bx;\by) \mc{L}(\bx;\by)  \prior(\bx)  d\bx \bs{C}_{\prior}\notag\\
&\quad + \inv{\ev[\by]} \bs{C}_{\prior} \Int_{\R^n}\nabla J(\bx;\by) \nabla J(\bx;\by)^\top \mc{L}(\bx;\by) \prior(\bx)d\bx + \bs{C}_{\prior} \notag\\
= \ & \bs{C}_{\prior} \left(-\E_{\post}(\nabla^2 J(X;\by)) + \E_{\post}(\nabla J(X;\by)\nabla J(X;\by)^\top) \right)\bs{C}_{\prior} + \bs{C}_{\prior}.
\end{align}
Finally, 
\begin{align}
\var_{\post}(X) &=  \E_{\post}( (X-\bs{\mu}_{\prior})(X-\bs{\mu}_{\prior})^\top ) - \E_{\post}(X-\bs{\mu}_{\prior})\E_{\post}(X-\bs{\mu}_{\prior})^\top \notag\\
&= \bs{C}_{\prior} \left(\var_{\post}(\nabla J(X;\by)) - \E_{\post}(\nabla^2 J(X;\by))\right)\bs{C}_{\prior} + \bs{C}_{\prior}
\end{align}
Using this relation, the condition~$\var_{\post}(X) \preceq \var_{\prior}(X) = \bs{C}_{\prior}$ is equivalent to 
\begin{equation} 
    \var_{\post}(\nabla J(X;\by)) \preceq \E_{\post}(\nabla^2 J(X;\by)).
\end{equation}
A sufficient condition is for instance
\begin{equation}
\E_{\post}(\nabla J(X;\by)\nabla J(X;\by)^\top) \preceq \E_{\post}(\nabla^2 J(X;\by)),
\end{equation}
which is true if,~$\forall X$,
\begin{equation}
\nabla J(X;\by)\nabla J(X;\by)^\top \preceq \nabla^2 J(X;\by).
\end{equation}
Considering this last condition is equivalent to consider that~$J$ must be convex and such that 
\begin{equation} 
    \scal{\nabla^2 J(X;\by)\nabla J(X;\by)}{\nabla J(X;\by)} \geq \norm{\nabla J(X;\by)}^4.
\end{equation}
These conditions are sufficient but not necessary, and still remains difficult to interpret.

\subsection{Counter examples}\label{subappendix:reduc-variance:examples}
Some cases where the assumption of reduced posterior variance is not verified are presented and illustrated in Figure~\ref{fig:reduc-variance:counter-example}.
\begin{enumerate}
\item \textbf{Figure~\ref{fig:reduc-variance:counter-example} (left)} - The \gls{blg} case satisfies the assumption, however, in the very special case of shifted distributions, the ratio of the covariance is not sufficient to find the directions. Let us consider a one-dimensional parameter space~$x$ endowed with a standard normal prior and a one-dimensional data~$y=30$. Setting the linear operator~$F = 0.05$ and the observation error~$\sigma_\varepsilon^2 = 1$, we get, using the expression detailed in Example~\ref{ex:blg}~$\bs{C}_{\post} = 0.9975 \simeq 1$ and~$\bs{\mu}_{\post} = 1.5 \neq 0$. While the posterior variance is really close to the prior one, the posterior mean is shifted, so that the direction is informed by the data. 
\item \textbf{Figure~\ref{fig:reduc-variance:counter-example} (middle)} - In the case where the forward operator is quadratic, a unique data could correspond to two parameter values. In that case, the posterior can become bimodal and can see its variance increasing. For this illustration,~$\mc{F}(x)~=~x^2$,~$y=10$ and~$\sigma_\varepsilon^2 = 10$.
\item \textbf{Figure~\ref{fig:reduc-variance:counter-example} (right)} - In the case of Gaussian priors, obtaining a larger variance \textit{a posteriori} is not so common but is possible in extreme cases where the data is not plausible according to the prior and the observation noise considered in the likelihood is important. Here, while the forward model is~$\mc{F}(x) = x^3$, the observed data~$y$ is set to~$200$ and~$\sigma_\varepsilon^2 = 2,000$.
\end{enumerate}
Note that these illustrations are obtained with Gaussian priors. With the use of other shapes for the prior, this situation of increasing variance is more common. 
\begin{figure}[ht]
    \subfloat{\includegraphics[width=0.3\textwidth]{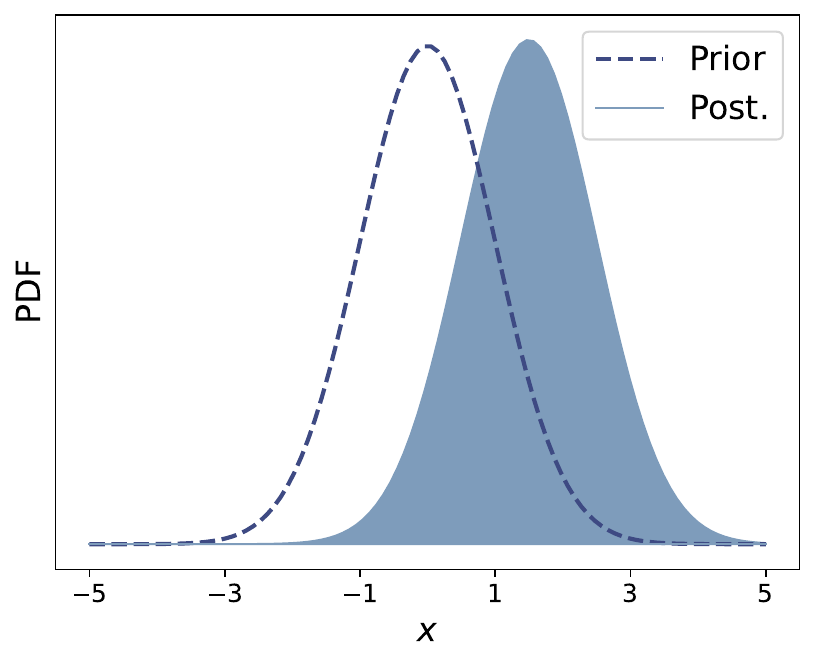}}\hfill
    \subfloat{\includegraphics[width=0.3\textwidth]{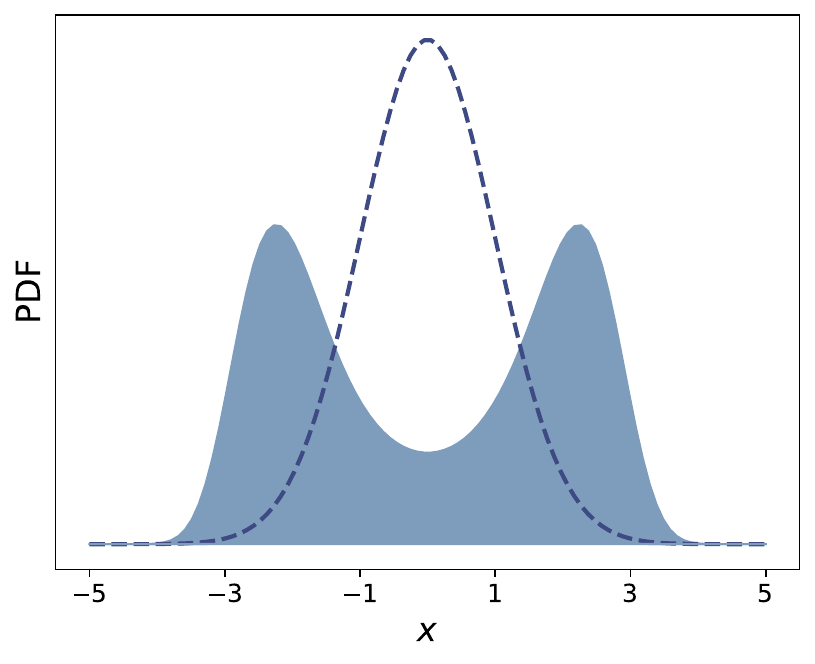}}\hfill
    \subfloat{\includegraphics[width=0.3\textwidth]{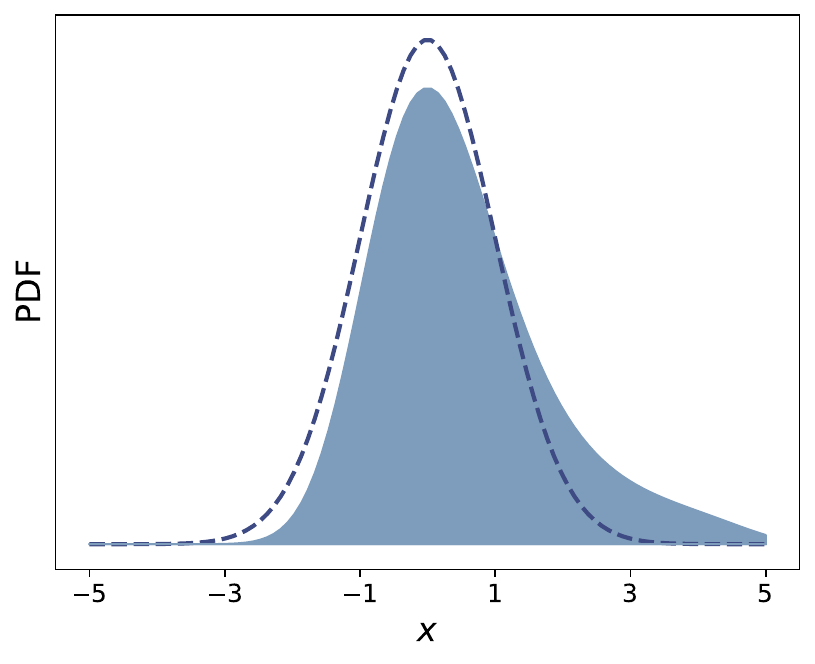}}
\caption{Illustration of inference cases where the posterior variance is not smaller than the prior one. (red plain) Prior p.d.f. (blue dashed) Posterior p.d.f. (left) Shifted distribution; (middle) Bimodal distribution; (right) Not so informative data.}\label{fig:reduc-variance:counter-example}
\end{figure}